\documentclass[10pt,hyp,english]{amsart}

\usepackage{amsmath, amscd, amsthm} 
\usepackage{amssymb, amsfonts, latexsym} 

\usepackage{booktabs}
\usepackage{mathtools}
\usepackage{enumerate}

\usepackage{color}

\usepackage[margin=1.1in]{geometry}
\usepackage{palatino}
\usepackage{mathrsfs}
\usepackage{booktabs}
\usepackage[svgnames]{xcolor} 
\usepackage{url}
\usepackage{breakurl}
\usepackage[breaklinks]{hyperref}

\usepackage{caption}
\usepackage{array,multirow}

\usepackage[T1]{fontenc}

\usepackage{tikz}
\usepackage{tikz-cd}
\usepackage{array}
\usepackage[backend=biber,giveninits=true,hyperref,style=alphabetic,maxcitenames=8,maxbibnames=8]{biblatex}
\usepackage{bbm}
\DeclareFieldFormat{postnote}{#1}
\DeclareFieldFormat{multipostnote}{#1}
\bibliography{main.bib}

\usepackage{hyperref}

\newcommand{\A}{\mathbb{A}}
\newcommand{\E}{\mathbf{E}}
\newcommand{\Z}{\mathbb{Z}}

\newcommand{\Q}{\mathbb{Q}}

\renewcommand{\P}{\mathbb{P}}
\renewcommand{\S}{\mathbf{S}}

\newcommand{\OFS}{\mathcal{O}_{F,\mathcal{S}}}
\newcommand{\OO}{\mathcal{O}}
\newcommand{\Gm}{\mathbb{G}_m}
\newcommand{\AAA}{\mathcal{A}}
\newcommand{\MGL}{\mathbf{MGL}}
\newcommand{\KGL}{\mathbf{KGL}}
\newcommand{\kgl}{\mathbf{kgl}}

\newcommand{\BPR}{BP\mathbb{R}}
\newcommand{\MU}{\mathbf{MU}}
\newcommand{\MO}{\mathbf{MO}}
\newcommand{\MZ}{\mathbf{M}\Z}
\newcommand{\MA}{\mathbf{M}A}
\newcommand{\C}{\mathbb{C}}
\newcommand{\RC}{R_\C^{C_2}}
\newcommand{\HZ}{\mathbf{H}\Z}
\newcommand{\R}{\mathbb{R}}
\newcommand{\SH}{\mathcal{SH}}
\newcommand{\Top}{\mathcal{T}op}
\newcommand{\Sm}{\operatorname{Sm}}

\newcommand{\Ext}{\operatorname{Ext}}
\newcommand{\Tor}{\operatorname{Tor}}
\newcommand{\colim}{\operatorname{colim}}
\newcommand{\cofib}{\operatorname{cofib}}
\newcommand{\holim}{\operatorname{holim}}
\newcommand{\tensor}{\otimes}
\newcommand{\Spec}{\operatorname{Spec}}
\newcommand{\Pic}{\operatorname{Pic}}
\newcommand{\Th}{\operatorname{Th}}

\newcommand{\Sq}{\operatorname{Sq}}
\newcommand{\wt}[1]{\widetilde{#1}}
\newcommand{\ol}[1]{\bar{#1}}
\newcommand{\et}{\acute{e}t}

\newcommand{\coker}{\operatorname{coker}}
\newcommand{\pr}{\operatorname{pr}}
\newcommand{\Gr}{\operatorname{Gr}}
\newcommand{\rk}{\operatorname{rk}}

\newcommand{\s}{\mathsf{s}}
\newcommand{\f}{\mathsf{f}}

\newcommand{\eff}{\text{eff}}
\newcommand{\veff}{\text{veff}}
\newcommand{\vf}{\widetilde{\mathsf{f}}}
\newcommand{\vcd}{\text{vcd}}
\newcommand{\cd}{\text{cd}}
\renewcommand{\star}{{\ast,\ast}}

\newcommand{\BPGL}{\mathbf{BPGL}}
\newcommand{\hocolim}{\operatorname{hocolim}}

\makeatletter
\def\iddots{\mathinner{\mkern1mu\raise\p@
    \vbox{\kern7\p@\hbox{.}}\mkern2mu
    \raise4\p@\hbox{.}\mkern2mu\raise7\p@\hbox{.}\mkern1mu}}
\makeatother

\title{Algebraic cobordism of number fields}

\usepackage[poorman]{cleveref}

\theoremstyle{theoremstyle}
\newtheorem{theorem}{Theorem}[section]
\newtheorem*{theorem*}{Theorem}
\newtheorem{lemma}[theorem]{Lemma}
\newtheorem{proposition}[theorem]{Proposition}
\newtheorem*{proposition*}{Proposition}
\newtheorem{corollary}[theorem]{Corollary}
\newtheorem*{corollary*}{Corollary}
\newtheorem{remark}[theorem]{Remark}
\newtheorem{remark*}{Remark}

\newtheorem{defn*}{Definition}
\theoremstyle{definition}

\theoremstyle{theoremstyle}

\thanks{The author was partially supported by the RCN Frontier Research Group Project no.~250399.}

\begin{document}

\title{Algebraic cobordism of number fields}
\subjclass[2010]{14F42 (primary),
  55N22, %
  55P91, %
  11R42, %
  55T99 %
  (secondary)}
\keywords{Motivic homotopy theory,
algebraic cobordism,
slice filtration,
Morava $K$-theory,
special values of Dedekind $\zeta$-functions of number fields
}

\author{Jonas Irgens Kylling}
\address{Department of Mathematics, University of Oslo, Norway}
\email{jonasik@math.uio.no}
\begin{abstract}
We compute the motivic homotopy groups of algebraic cobordism over number fields,
the motivic homotopy groups of 2-complete algebraic cobordism over the real numbers and rings of $2$-integers
and the motivic homotopy groups of mod 2 motivic Morava $K$-theory over fields with low virtual cohomological dimension.
As an application we relate the order of the algebraic cobordism groups of rings of 2-integers to special values of Dedekind $\zeta$-functions of totally real abelian number fields.
\end{abstract}
\maketitle

\section{Introduction}

Algebraic cobordism is a bigraded cohomology theory of smooth schemes represented by the Thom spectrum $\MGL$ in the stable motivic homotopy category.
Algebraic cobordism was introduced by Voevodsky as an analogue to complex cobordism to help solve the Milnor conjecture \cite{Voevodsky:ICM}. This paper investigates the bigraded motivic homotopy groups $\pi_\star(\MGL)$ of the algebraic cobordism spectrum over the real numbers, number fields and rings of $\mathcal S$-integers.

The algebraic cobordism spectrum $\MGL$ is a $\P^1$-spectrum constructed analogously to the complex cobordism spectrum $\MU$ in topology \cite[6.3]{Voevodsky:ICM}.
Consider the Grassmann scheme $\Gr(m,n)$ of dimension $m$-planes in $\A^n$, and its canonical $m$-vector bundle $\gamma_{m,n}$.
Taking the colimit over $n$ we get the infinite Grassmannian $\Gr(m,\infty)$
with canonical bundle $\gamma_{m,\infty}$.
The embeddings $\Gr(m,\infty) \to \Gr(m+1, \infty)$ induce maps $\epsilon^1\oplus\gamma_{m,\infty} \to \gamma_{m+1,\infty}$, where $\epsilon^1$ is the trivial rank one bundle. In the unstable motivic homotopy category the infinite Grassmannians are classifying spaces for the isomorphism classes of $m$-vector bundles over smooth schemes \cite{Morel:A1}.
Taking the Thom spaces of the bundles $\gamma_{m,\infty}$ we get the constituent spaces of the $\P^1$-spectrum $\MGL = (\Th(\gamma_{0,\infty}), \Th(\gamma_{1,\infty}), \dots)$ with structure maps
\[
\P^1 \wedge \Th(\gamma_{m,\infty})
\cong \Th(\epsilon^1\oplus\gamma_{m,\infty})
\to \Th(\gamma_{m+1,\infty}).
\]

In motivic homotopy theory the spheres are bigraded
$S^{p,q} = (S^1_s)^{p-q}\wedge \Gm^{\wedge q}$, where $S^1_s$ is the simplicial circle and $\Gm$ is the punctured affine line pointed at 1.
Hence motivic homotopy groups are bigraded as well, similar to how $C_2$-equivariant homotopy groups are indexed by the trivial representation and the sign representation.
The bigraded homotopy groups of the algebraic cobordism spectrum are known along some special lines.
The Hopkins-Morel isomorphism implies that for fields in characteristic 0 there is an isomorphism $\pi_{2n,n}(\MGL) \cong L_n$,
where $L_*$ is the Lazard ring \cite[Proposition 8.2]{Hoyois:fromto}.
More generally Levine and Morel \cite{Levine-Morel} defined a cohomology theory $\Omega^*(-)$ on smooth schemes in geometric terms such that for a smooth scheme $X$ there is an isomorphism $\Omega^n(X)\cong \MGL^{2n,n}(X)$ \cite[Corollary 8.15]{Hoyois:fromto}.
It is an open problem to give a geometric description of all of $\MGL^\star(X)$.
For a field $F$ Spitzweck showed that $\MGL_{2n+1,n}(F;\Z) \cong F^\times \tensor L_{n+1}$ \cite[Corollary 7.5]{Spitzweck:mixed}, because the slice spectral sequence collapses in this range.
For the same reason $\MGL_{2n+2,n}(F;\Z) \cong K_2(F) \tensor L_{n+2}$,
and along the diagonal there is an isomorphism with Milnor $K$-theory $ K^M_{-n}(F) \cong \pi_{n,n}(\MGL)$.
Furthermore, $\pi_{p,q}(\MGL) = 0$ for $p < q$ or $2p < q$.
We compute all of the bigraded motivic homotopy groups of the algebraic cobordism spectrum over number fields in terms of the integral motivic cohomology of the number fields, up to extension, in \Cref{thm:MGLF}.
Over rings of $\mathcal{S}$-integers and the real numbers we compute the homotopy groups of the 2-completed algebraic cobordism spectrum in \Cref{thm:OFS} and \Cref{thm:MGLR}.
As an application we relate the order of
the algebraic cobordism groups of the ring of $2$-integers in a totally real abelian number field $F$
to special values of the Dedekind $\zeta$-function of $F$ in \Cref{thm:zeta}.

Associated to the algebraic cobordism spectrum there is the motivic Brown-Peterson spectrum $\BPGL$, its truncations $\BPGL\langle n \rangle$, and motivic Morava $K$-theory $K(n)$.
The topological realizations of these motivic spectra play an important role in chromatic homotopy theory.
The homotopy groups of these motivic spectra can be computed by the same techniques used for $\MGL$.
We compute $\pi_\star(\BPGL)$, $\pi_\star(\BPGL\langle n\rangle)$ over the real numbers in \Cref{thm:BPGL}, and $\pi_\star(K(n))$ over fields with virtual cohomological dimension less than $2^{n+1}-2$ in \Cref{thm:Kn}.

Our main tool is the slice spectral sequence and the slice filtration introduced by Voevodsky in \cite{Voevodsky:open}.
We also make use of the interplay between motivic and $C_2$-equivariant stable homotopy theory. This interplay is facilitated by the $C_2$-equivariant complex realization functor induced by sending a smooth scheme over $\R$ to its complex realization equipped with the $C_2$-action given by complex conjugation.
This allows us to use results of Hill, Hopkins and Ravenel \cite{HHR} in the motivic setting. Some of our proofs are both inspired by and heavily dependent on those of \cite{HHR}.

The results over rings of $\mathcal S$-integers relies on the work of Levine \cite{Levine99} and Spitzweck \cite{Spitzweck:commutative} on motivic cohomology of Dedekind domains. Our results over rings of integers are somewhat similar to those of Rognes and Weibel \cite{Rognes-Weibel} on algebraic $K$-theory of rings of 2-integers. They use the Bloch-Lichtenbaum spectral sequence to compute 2-primary algebraic $K$-theory of number fields and rings of $2$-integers. 

\subsection*{Previous work}
Yagita \cite{Yagita:atiyah} uses the slice spectral sequence to compute the associated graded of $\pi_{*,*}(\BPGL/2)$ over the real numbers.
Even earlier Hu and Kriz \cite{Hu-Kriz:real} computed the coefficients $\pi_{*,*}^{C_2}(\BPR)$ of the Brown-Peterson spectrum associated to $C_2$-equivariant complex cobordism. A careful proof of this is also in the appendix of \cite{Greenlees-Meier}.
Their answer has essentially the same form as $\MGL_{*,*}(\R;\Z_2)$,
and they explore relations to $\MGL$ in \cite{Hu-Kriz:remarks},
where they also compute $K_\star(n)(\R)$.
In \cite{Hill} Hill proves that the motivic Adams spectral sequence of $\BPGL\langle n \rangle$ over $\R$ collapses, and notes that $\pi_\star(\BPGL)$ can be obtained from the computation of $\pi_\star^{C_2}(BP\R)$ in \cite{Hu-Kriz:real}.
Ormsby uses the motivic Adams spectral sequence to compute $\pi_\star(\BPGL\langle n \rangle)$ over $p$-adic fields with the motivic Adams spectral sequence in \cite{Ormsby}.
Similar techniques are used by Ormsby and \O{}stv\ae{}r \cite{OP:BP} to compute $\pi_{*,*}(\BPGL\langle n \rangle)$ and $\pi_{*,*}(\MGL_2^{\wedge})$ over the rational numbers.
Ellis considers a motivic $C_2$-equivariant version of $\MGL$ and computes its coefficient over the complex numbers in his thesis \cite{Ellis}. %
Recently Heard \cite{Heard} compared the motivic slice filtration of $\MGL$ to the equivariant slice filtration of $\MU$, and computed $\pi_\star(\BPGL)$ and $\pi_\star(\BPGL\langle n \rangle/2)$ over the real numbers.

\subsection*{Organization of this paper}
We begin with three short sections recalling theory and results used later in the paper.
In \Cref{sec:intro} we give a quick review of the relation between motivic and $C_2$-equivariant stable homotopy theory.
\Cref{sec:motcoh} follows with a description of (2-complete) motivic cohomology of the real numbers, number fields and rings of $\mathcal{S}$-integers in number fields.
In \Cref{sec:slice} we review the slice spectral sequence with particular attention to the case of $\MGL$.
Next are the computations, which are the most technical part of the paper.
We compute $\MGL_{*,*}(\R;\Z_2)$ in \Cref{sec:MGLR} by first running the slice spectral sequence for $\MGL/2^n$ and then passing to the limit over $n$.
We determine the multiplicative extensions in $\MGL_{*,*}(\R;\Z_2)$ by elementary means in contrast to \cite{Hu-Kriz:real}, cf.~\Cref{rmk:elementary}. This is similar to the proof in \cite{Greenlees-Meier}.
In \Cref{sec:MGLF} we use the computations over $\R$ to determine the differentials and the $E^\infty$-page of the slice spectral sequence over number fields.
In \Cref{sec:zeta} we relate the cardinality of $\MGL_{*,*}(\OO_F[\frac{1}{2}];\Z_2)$ to special values of the Dedekind $\zeta$-function of a totally real abelian number field $F$.
This is a corollary of a result of Manfred Kolster and the computation of $\MGL_{*,*}(\OO_F[\frac{1}{2}];\Z_2)$ in terms of motivic cohomology.
For completeness we summarize some easy results on the motivic homotopy groups of (2-completed) algebraic cobordism over fields of $2$-cohomological dimension less than or equal to $2$ in \Cref{sec:final}. In this case the slice spectral sequence collapses for degree reasons.
In \Cref{sec:morava} we compute the motivic homotopy groups of the (truncated) Brown-Peterson spectra $\BPGL$ and $\BPGL\langle n\rangle$ over the real numbers, and of Morava K-theory $K(n)$ over fields with virtual cohomological dimension $\vcd(F) < 2(2^n - 1)$.
The techniques are the same as for the algebraic cobordism spectrum.

\subsection*{Acknowledgments}
I would like to thank Lorenzo Mantovani for interesting discussions which inspired the questions considered in this paper.
I would also like to thank John Rognes, Oliver R\"{o}ndigs, Glen Wilson and Paul Arne \O{}stv\ae{}r for helpful discussions and comments.
The author is grateful to Lennart Meier for pointing out an error in the statement of \Cref{thm:MGLR}. The author thanks the Hausdorff research Institute for Mathematics in Bonn for the hospitality and the delicious cakes during the Hausdorff Trimester Program on ``K-Theory and Related Fields'' in the summer of 2017, where this work was initiated. The author is also grateful for the hospitality of the Faculty of Mathematics of the University of Duisburg-Essen for a week during the autumn of 2017.

\subsection*{Notation}
Throughout the paper we will work over a number field $F$, with the exception of the final section on motivic Morava $K$-theory.
We write $\OFS$ for the ring of $\mathcal S$-integers in $F$, for $\mathcal S \supset \{2, \infty \}$ a set of places.
We will work in the stable motivic homotopy category $\SH(F)$.
The following table summarizes the notation used in the paper:
\begin{center}
\footnotesize
\begin{tabular}{l|l}
$\SH(F)$, $\SH^{C_2},\SH$ & motivic, $C_2$-equivariant and ordinary stable homotopy category \\
$\E$, $\S$  & motivic spectrum, motivic sphere spectrum \\
$\MGL$, $\MZ$ & algebraic cobordism, motivic cohomology spectrum \\
$\BPGL$, $\BPGL\langle n \rangle$ & motivic (truncated) Brown-Peterson spectrum \\
$K(n)$ & $n$th motivic Morava $K$-theory \\
$\E_2^{\wedge}$, $\MZ_2$ & $\holim_n \E/2^n$, $\holim_n \MZ/2^n$ \\
$\pi_{p,q}(\E)$ & $[(S^1)^{\wedge p-q}\wedge \Gm^{\wedge q}, \E]_{\SH(F)}$ \\
$\E_{*,*}(F;\Z)$, $\E_{*,*}(F;\Z/2^n)$ & $\pi_{*,*}(\E)$, $\pi_{*,*}(\E/2^n)$ \\
$\E_{*,*}(F;\Z_2)$ &
$\pi_{*,*}(\E_2^{\wedge})$ \\
$H^{p,q}(F;A)$ & $\pi_{-p,-q}(\MA)$, $A \neq \Z_2$ an abelian group \\ %
$\wt{H}^{p,q}(F;\Z)$ & $\ker(H^{p,q}(F;\Z) \to \oplus^{r_1} H^{p,q}(\R;\Z))$, $F$ a field with $r_1$ real embeddings \\
$\f_q$, $\s_q$ & $q$th effective cover, $q$th slice \\
$E^r(F;\Z)$, $E^r(F;\Z/2^n)$ & $E^r$-page of the slice spectral sequence of $\MGL$ and $\MGL/2^n$ over $F$ \\
$E^r(F;\Z_2)$ & $\lim_n E^r(F;\Z/2^n)$ \\
$E^r$, $E^r(\Z)$, $E^r(\Z/2^n)$ & $E^r(F;\Z)$, and so on \\
$r_1$, $r_2$ & real and complex conjugate pairs of embeddings of a number field $F$ \\
$\OFS$ & ring of $\mathcal{S}$-integers in a number field $F$, for a set of places $\mathcal{S} \supset \{2,\infty\}$ \\
$\nu_2(n)$ & the $2$-adic valuation of $n$ \\
$L_*$, $K^M_*(F), k_*(F)$ & the Lazard ring, integral and mod 2 Milnor $K$-theory of $F$ \\
$R_\C$, $R_\C^{C_2}$, $R_\R$,  & complex realization, $C_2$-equivariant complex realization, real realization \\
$\Phi^{C_2}$, $\rho$, $\sigma$ & geometric fixed points, $\{\pm 1 \} = S^{0} \hookrightarrow \Gm \in \pi_{-1,-1}(\S)$, sign representation of $C_2$ \\
$\AAA^\star, \AAA_\star$, $\Delta$ & motivic Steenrod algebra, dual, coproduct of dual \\
$\tau$, $\rho$, $\xi_i$, $\tau_i$ & algebra generators of $\AAA_\star(\R)$ \\
$\Sq^i$, $Q_i$ & motivic Steenrod squares, Milnor primitive, the dual of $\tau_i$ \\
$\vcd(F)$ & $\cd_2(F(\sqrt{-1}))$, virtual cohomological dimension of a field $F$
\end{tabular}
\end{center}

\section{Motivic and $C_2$-equivariant stable homotopy theory}
\label{sec:intro}
In this section we recall the construction of some functors between the motivic, $C_2$-equivariant and topological stable homotopy category, and the image of some of the spectra we are interested in by these functors.
The main observation is \Cref{lem:fixed-points-realization}, which tells us how to compute real realizations.

Recall that there are realization functors
\begin{align*}
R_\C : \SH(\C) \to \SH, \qquad
R_\C^{C_2} : \SH(\R) \to \SH^{C_2}, \qquad
R_\R : \SH(\R) \to \SH.
\end{align*}
The $C_2$-equivariant complex realization functor $R_\C^{C_2}$ is induced by
\[
\Sm/\Spec(\R) \ni X \mapsto X(\C) \in \Top^{C_2}
\]
with the analytic topology and $C_2$-action given by complex conjugation,
while $R_\R$ is induced by
\[
\Sm/\Spec(\R) \ni X \mapsto X(\R) \in \Top.
\]
This give rise to well defined functors of the stable homotopy categories since the motivic stable homotopy category is generated by suspension spectra of schemes, the functors map motivic spheres to ($C_2$-equivariant) spheres, and are compatible with $\A^1$-invariance and Nisnevich descent.
Both realization functors are left adjoints.
See \cite{HO:galois} for a more careful construction and thorough discussion of the realization functors.

Recall that the geometric fixed points functor $\Phi^{C_2} : \SH^{C_2} \to \SH$ is characterized by the following properties \cite[Remark 7.15]{Schwede}, \cite[Proposition 2.45]{HHR}:
\begin{enumerate}
\item $\Phi^{C_2}(\Sigma^\infty_+ A) = \Sigma^\infty_+ (A^{C_2})$,
\item $\Phi^{C_2}$ is monoidal,
\item $\Phi^{C_2}$ preserves filtered homotopy colimits.
\end{enumerate}
Hence
\[
R_\R\Sigma_+^\infty (X) = \Sigma_+^\infty (X(\R)) = \Sigma_+^\infty (X(\C)^{C_2}) = \Phi^{C_2}\Sigma_+^\infty(X(\C)) = \Phi^{C_2}R_\C^{C_2}\Sigma^\infty_+(X).
\]
Since $\Phi^{C_2}$ preserves filtered colimits and $\SH(\R)$ is generated by suspension spectra of schemes we get the following lemma, which seems to be folklore.
\begin{lemma}
\label{lem:fixed-points-realization}
Real realization is naturally isomorphic to the geometric fixed points of the $C_2$-equivariant complex realization. That is,
$
R_\R \cong \Phi^{C_2}R_\C^{C_2}.
$
\end{lemma}
Bachmann proved an equivalence of $\SH(\R)[\rho^{-1}]$ with $\SH$.
\begin{theorem}[\protect{\cite{Bachmann:real}}]
\label{thm:tom}
Real realization induces an equivalence
\[
\SH(\R)[\rho^{-1}] \to \SH.
\]
\end{theorem}
To summarize, we have a commutative square
\[
\begin{tikzcd}
\SH(\R) \ar[r, "R_\C^{C_2}"]\ar[d] & \SH^{C_2} \ar[d, "\Phi^{C_2}"] \\
\SH(\R)[\rho^{-1}] \ar[r, "R_\R"] & \SH
\end{tikzcd}  
\]
where the bottom horizontal map is an equivalence.

We have $R^{C_2}_\C(\MGL) = \MU$, where $\MU$ is considered as the $C_2$-equivariant complex cobordism spectrum defined by  Landweber \cite{Landweber}.
Indeed, $\MGL \simeq \colim_q \Sigma^{-2q,-q}\Th(\gamma_{q,\infty}),$
and the $C_2$-equivariant complex realization of $\Th(\gamma_{q,\infty})$ are the Thom spaces defining $\MU$ as a $C_2$-spectrum.
Heller and Ormsby showed that the $C_2$-equivariant complex realization of $\MZ$ is $C_2$-equivariant Bredon-cohomology \cite[Theorem 4.17]{HO:galois}.
That is, $\RC(\MZ) = \HZ$, for $\Z$ the constant Mackey functor.

\section{Motivic cohomology of the real numbers and number fields}
\label{sec:motcoh}
In this section we state some results on the structure of motivic cohomology of the real numbers, number fields and rings of $\mathcal{S}$-integers in number fields. For a number field $F$ with $r_1$ real embeddings, let $\mathcal S \supset \{2, \infty\}$ be a (not necessarily finite) set of places in $F$. We denote the ring of $\mathcal S$-integers by $\OFS$.

The Bloch-Kato-conjecture (\cite{Voevodsky:Z2}, \cite{Voevodsky:Zl}) gives an isomorphism of motivic cohomology and
\'{e}tale-cohomology with finite coefficients above the diagonal.
That is,
\begin{equation}
\label{eq:BK}
H^{p,q}(F;A) \cong H^{p,q}_{\et}(F; A), \text{ when } p \leq q,
\end{equation}
and $A$ a finite abelian group.
We have, cf.~\cite[2.1]{DI:adams},
\[
H^{*,*}(\R;\Z/2) \cong \Z/2[\rho,\tau],
\]
where $\rho = [-1] \in \R^\times/(\R^\times)^2 \cong H^{1,1}(\R;\Z/2)$,
and $\tau = [-1] \in \mu_2(\R) \cong H^{0,1}(\R;\Z/2)$.
\begin{lemma}
\label{lem:HR}
Let $n > 1$.
As a $\Z/2^n$-algebra mod $2^n$ motivic cohomology of the real numbers is
\[
H^{*,*}(\R;\Z/2^n) \cong \Z/2^{n}[\rho,\tau, u]/(2\rho, 2\tau, \tau^2).
\]
The elements $\rho$, $\tau$ and $u$ are the generators of the cohomology groups in bidegrees $(1,1)$, $(0,1)$ and $(0,2)$, represented by $-1$, $-1$ and $\zeta_{2^n}$, respectively.
As a $\Z_2$-algebra $2$-complete motivic cohomology of the real numbers is
\[
H^{*,*}(\R;\Z_2) \cong \Z_2[\rho, u]/(2\rho).
\]
\end{lemma}
Note that $\rho$ maps to $\rho$ and $u$ maps to $\tau^2$ via the projection $H^{*,*}(\R;\Z/2^n) \to H^{*,*}(\R;\Z/2)$.
The element $\tau \in H^{*,*}(\R;\Z/2^n)$ is the image of $\tau \in H^{*,*}(\R;\Z/2)$ through $H^{*,*}(\R;\Z/2) \to H^{*,*}(\R;\Z/2^n)$. We use the same name for $\tau$ and $\rho$ in the various motivic cohomology groups.
\begin{proof}
With \eqref{eq:BK} this reduces to a computation in \'{e}tale cohomology.
That is, we must compute
\[
H^{p}_{\et}(\R;\mu_{2^n}^{\tensor q}) = \Ext^p_{\Z[C_2]}(\Z, \mu_{2^n}^{\tensor q}),
\] where $\Z[C_2] = \Z[x]/(x^2-1)$.
Use the resolution
\[
\dots \xrightarrow{\cdot (x+1)} \Z[C_2] \xrightarrow{\cdot (x-1)} \Z[C_2] \xrightarrow{x \mapsto 1} \Z,
\]
and that $x$ acts on $\mu_{2^n}^{\tensor q} = \Z/2^n$ as $1$ when $q$ is even and $-1$ when $q$ is odd.
Products are formed by forming tensor products of functions.

To obtain the $2$-adic description we take the limit.
The structure maps in the inverse system are induced by the projection maps of the coefficients.
The structure map of $H^{p}_{\et}(\R;\mu_{2^n}^{\tensor q})$ is multiplication by $2$ when $p - q$ is odd. Hence, all multiples of $\tau$ vanish in the limit.
\end{proof}

\begin{theorem}[\protect{\cite[Theorems 14.5, 14.6]{Levine99}}]
\label{thm:hOFS}
Let $\OFS$ be the ring of $\mathcal{S}$-integers,
$\mathcal{S} \supset \{2, \infty\}$,
in a number field $F$ with $r_1$ real embeddings.
Let $\Z_{\mathcal S}$ be the localization of $\Z$ such that the primes not in $\mathcal S$ are invertible.
The motivic cohomology groups $H^{p,q}(\OFS;\Z_{\mathcal S})$ are trivial outside the range $1 \leq p \leq q$ except possibly in the bidegrees $(0,0)$ and $(2,1)$.
We have $H^{0,0}(\OFS;\Z_{\mathcal S}) = \Z_{\mathcal S}$ and $H^{2,1}(\OFS;\Z_{\mathcal S}) = \Pic(\OFS)\tensor\Z_{\mathcal S}$.
For $3 \leq p\leq q$ the real embeddings induce an isomorphism
\[
H^{p,q}(\OFS;\Z_{\mathcal S})
\cong
\oplus^{r_1} H^{p,q}(\R;\Z_2)
\cong
\begin{cases}
(\Z/2)^{r_1} & p \equiv q \bmod 2 \\
0 & p \not\equiv q \bmod 2.
\end{cases}
\]
\end{theorem}
\begin{proof}[Proof sketch]
We only prove the final statement for number fields following \cite[Theorem 14.5]{Levine99}.
By the Beilinson-Lichtenbaum conjecture \cite[Theorem 6.1]{Voevodsky:Z2}, \cite[Theorem 6.17]{Voevodsky:Zl} there is an isomorphism
$
H^{p,q}(F;\Z/\ell) \cong H^{p}_{\et}(F;\mu_{\ell}^{\tensor q}).
$
A number field $F$ has $\ell$-cohomological dimension 2 when $\ell$ is odd \cite[Proposition 8.3.17]{NSW}.
When $\ell = 2$ we have an isomorphism $H^{p,q}(F;\Z/2) \cong K^M_{p}/2\{\tau^q \} \cong  \Z/2$ for $q \geq p \geq 3$ by \cite[Theorem A.2]{Milnor:K-theory}.
Levine shows $H^{p,q}(F;\Q) = 0$ for $q \neq 1$, except for $p=q=0$ \cite[Theorem 14.5]{Levine99}. Combining the torsion part and the rational part implies the statement.
\end{proof}
More detailed descriptions of $H^{1,q}(\OFS;\Z_{\mathcal S})$ and $H^{2,q}(\OFS;\Z_{\mathcal S})$ are given in \cite[14]{Levine99}.

\begin{corollary}
\label{lem:H3R}
Let $F$ be a number field.
Then for $p \geq 3$ the canonical maps
\[
H^{p,q}(F;\Z_{\mathcal S})
\xrightarrow{\cong} H^{p,q}(F;\Z_2)
\xrightarrow{\cong} \oplus^{r_1} H^{p,q}(\R;\Z_2)
\]
are isomorphisms.
\end{corollary}

\begin{lemma}
\label{lem:HFR}
\label{lem:H12R}
Let $\OFS$ be a number field or its ring of $\mathcal{S}$-integers.
Then the canonical map
\[
H^{2,q}(\OFS;\Z_{\mathcal S}) \twoheadrightarrow \oplus^{r_1} H^{2,q}(\R;\Z_2)
\]
induced by the real embeddings
is surjective.
\end{lemma}
\begin{proof}
By \Cref{lem:HR} we may assume $q$ is even.
Consider the commutative diagram
\[
\begin{tikzcd}
H^{2,q}(\OFS;\Z_{\mathcal S}) \ar[r]\ar[d, "\pr"] & \oplus^{r_1}H^{2,q}(\R;\Z_2)\ar[d, "\pr"] \\
H^{2,q}(\OFS;\Z/2) \ar[r]\ar[d, "\rho"] & \oplus^{r_1}H^{2,q}(\R;\Z/2) \ar[d, "\rho"] \\
H^{3,q+1}(\OFS;\Z/2) \ar[r] & \oplus^{r_1}H^{3,q+1}(\R;\Z/2).
\end{tikzcd}
\]
The vertical maps in the right column are isomorphisms by \Cref{lem:HR}.
The bottom horizontal map is an isomorphism by \Cref{thm:hOFS}. %
Multiplication by $\rho$ induces a surjective map $H^{2,q}(\OFS;\Z/2) \to H^{3,q+1}(\OFS;\Z/2)$ by \cite[Lemma 7.18]{KRO}.
The projection $\pr: H^{2,q}(\OFS;\Z_{\mathcal S}) \to H^{2,q}(\OFS;\Z/2)$ has cokernel $H^{3,q}(\OFS;\Z_{\mathcal S}) = 0$.
Hence the top horizontal map is surjective.
\end{proof}
\begin{remark}
In general, the map $H^{1,q}(\OFS;\Z_{\mathcal S}) \to \oplus^{r_1} H^{1,q}(\R;\Z_2)$ has a nontrivial cokernel,
cf.~\cite{Rognes}.
However, the cokernels are isomorphic for $q \geq 1$ of the same parity.
Indeed, for $q$ odd this follows from the commutative diagram 
\[
\begin{tikzcd}
H^{1,q}(\OFS;\Z_{\mathcal S}) \ar[r]\ar[d, "\pr"] & \oplus^{r_1}H^{1,q}(\R;\Z_2)\ar[d, "\pr"] \\
H^{1,q}(\OFS;\Z/2) \ar[r]\ar[d, "\tau^2"] & \oplus^{r_1}H^{1,q}(\R;\Z/2) \ar[d, "\tau^2"] \\
H^{1,q+2}(\OFS;\Z/2) \ar[r] & \oplus^{r_1}H^{1,q+2}(\R;\Z/2) \\
H^{1,q+2}(\OFS;\Z_{\mathcal S}) \ar[r]\ar[u, "\pr"] & \oplus^{r_1}H^{1,q+2}(\R;\Z_2)\ar[u, "\pr"]
\end{tikzcd}
\]
since $\pr$ is an isomorphism over $\R$ and $\tau$ is an isomorphism on mod 2 motivic cohomology.
\end{remark}

\section{The slice spectral sequence of algebraic cobordism}
\label{sec:slice}
In this section we construct the slices spectral sequence of $\MGL$ and discuss its multiplicative and strong convergence properties.
The slice filtration was originally introduced by Voevodsky in \cite{Voevodsky:open}. In \cite{RO:hermitian}, \cite{RSO:April1} and \cite{KRO} there are slice spectral sequence computations similar in spirit to ours.
The slice filtration is obtained by considering the triangulated subcategory $\SH^{\eff}(F) \subset \SH(F)$ generated by suspension spectra of smooth schemes.
We then filter by $i_q : \Sigma^{2q,q}\SH^{\eff}(F) \hookrightarrow \SH(F), q \in \Z$. Each inclusion $i_q$ has a right adjoint $r_q$,
and the composite $\f_q = i_q \circ r_q$ is defined to be the $q$th effective  cover.
The cofiber $\s_q(-) = \cofib(\f_{q+1}(-) \to \f_q(-))$ is the $q$th slice functor,
and these functors assemble to a tower of cofiber sequences which gives rise to a spectral sequence in the usual way. Since $\SH^{\eff}(F)$ is a triangulated category, the functors $\f_q$ and $\s_q$ are exact.
It is possible to consider a finer filtration of $\SH(F)$ by considering the subcategory $\SH^\veff(F) \subset \SH^\eff(F)$ consisting of the spectra which are connective in Morel's $t$-structure \cite{Bachmann:veff}. For $\MGL$ these filtrations agree since $\f_q(\MGL)$ is $q$-connective, see \Cref{lem:veff-agree}. %

Recall that the slices of $\MGL$ are \cite{Voevodsky:open}, \cite[Theorem 4.7]{Spitzweck:slices}, \cite[Theorem 8.5]{Hoyois:fromto}
\[
\s_q(\MGL) = \Sigma^{2q,q}\MZ \tensor L_q,
\]
compatible with the ring map $L_* \to \MGL_{2*,*}$.
Here $L_* = \Z[x_1, x_2, \cdots]$ is the Lazard ring,
with generators $x_i$ in degree $i$ (i.e., half of the usual indexing).
Since the slices are modules over $\s_0(\S) = \MZ$ \cite[Theorem 3.6.22]{Pelaez} the multiplicative structure on the slices $\s_*(\MGL)$ is the one induced from the product on $\MZ$ and $L_*$.
The slice spectral sequence of $\MGL$ is obtained by taking homotopy groups in the slice tower of $\MGL$.
\[
\begin{tikzcd}
\dots \ar[r] & \f_{q+1}(\MGL) \ar[r]\ar[d] & \f_{q}(\MGL) \ar[r]\ar[d] & \f_{q-1}(\MGL) \ar[r]\ar[d] & \dots \ar[r] & \f_0(\MGL) = \MGL \\
& \s_{q+1}(\MGL) & \s_{q}(\MGL) & \s_{q-1}(\MGL) & 
\end{tikzcd}
\]
This spectral sequence has $E^1$-page
\[
E^1_{p,q,w} = \pi_{p,w}\s_q(\MGL),
\]
and $d^r$-differential
\[
d^r : E^r_{p,q,w} \to E^r_{p-1,q+r,w}.
\]

By the work of Pelaez the diagrams
\[
\begin{tikzcd}
\f_{q+1}(\MGL)\wedge \f_{q'}(\MGL) \ar[r]\ar[d] & \f_{q+q'+1}(\MGL) \ar[d] \\
\f_{q}(\MGL)\wedge \f_{q'}(\MGL) \ar[r] & \f_{q+q'}(\MGL)
\end{tikzcd}
\]
in $\SH(F)$ are induced by a zigzag of commutative diagrams in motivic symmetric spectra \cite[Theorem 3.6.16]{Pelaez}.
Hence the slice filtration gives rise to a multiplicative Cartan-Eilenberg system,
and the slice spectral sequence is a multiplicative spectral sequence, see also \cite{GRSO}.
For formal reasons the slice spectral sequence is conditionally convergent to the homotopy groups of the slice completion (\cite[8.4]{Hoyois:fromto}, \cite[Definition 3.1]{RSO:April1}) of $\MGL$.
Since $\MGL$ is slice complete \cite[Lemma 8.10, Corollary 2.4]{Hoyois:fromto}, i.e., $\MGL$ is equivalent to its slice completion, the slice spectral sequence is conditionally convergent to the homotopy groups of $\MGL$.
This is also true over rings of $\mathcal S$-integers for $\MGL$ localized at $\mathcal S$, $\MGL_{\mathcal S}$, see \cite{Spitzweck:mixed}.
Since in a fixed tridegree $(p,q,w)$ there are only finitely many entering and exiting differentials of $E^1_{p,q,w}$, the conditional convergence is in fact strong. That is, we have a strongly convergent spectral sequence
\[
E^1_{p,q,w} = \pi_{p,w}\s_q(\MGL) = H^{2q-p,q-w}(F;\Z)\tensor L_q \implies \pi_{p,w}(\MGL).
\]
See \Cref{fig:E1-page} for a picture of the spectral sequence.
Similarly we also have strong convergence of the slice spectral sequence for $\MGL/2^n$.
Its slices are
\[
\s_q(\MGL/2^n) = \Sigma^{2q,q}\MZ/2^n \tensor L_q.
\]

Consider the system of maps
\begin{equation}
\label{eq:2n-sys}
\begin{tikzcd}
\MGL \ar[r, "2^n"] & \MGL \ar[r] & \MGL/2^n \ar[r] & \Sigma^{1,0}\MGL \\
\MGL \ar[r, "2^{n+1}"]\ar[u, "2"] & \MGL \ar[r]\ar[u] & \MGL/2^{n+1} \ar[r]\ar[u] & \Sigma^{1,0}\MGL \ar[u, "2"]
\end{tikzcd}
\end{equation}
and the induced limit of spectral sequences
\[
E^r_{p,q,w}(\Z_2) = \lim_n E^r_{p,q,w}(\Z/2^n).
\]
In general we make no claims about $\{E^r_{p,q,w}(\Z_2)\}_r$ being a spectral sequence or its convergence.
However, if the groups $E^\infty_{p,q,w}(\Z/2^n)$ are finite
then $E^\infty_{p,q,w}(\Z_2)$ is the associated graded of an exhaustive, Hausdorff and complete filtration of $\pi_{*,*}(\MGL_2^{\wedge})$.
This is the case over the real numbers and rings of $\mathcal S$-integers when $\mathcal S$ is finite.
\begin{figure}
  
\begin{tikzpicture}[font=\tiny,scale=0.4]
{\node[above right=0pt] at (0,0.4) { $0,0$};}
{\draw (0,0) -- (2,0);}
{\draw (2,0) -- (2,2);}
{\draw (2,2) -- (0,2);}
{\draw (0,2) -- (0,0);}
{\node[above right=0pt] at (4,2.4) { $0,1$};}
{\draw (4,2) -- (6,2);}
{\draw (6,2) -- (6,4);}
{\draw (6,4) -- (4,4);}
{\draw (4,4) -- (4,2);}
{\node[above right=0pt] at (2,2.4) { $1,1$};}
{\draw (2,2) -- (4,2);}
{\draw (4,2) -- (4,4);}
{\draw (4,4) -- (2,4);}
{\draw (2,4) -- (2,2);}
{\node[above right=0pt] at (8,4.4) { $0,2$};}
{\draw (8,4) -- (10,4);}
{\draw (10,4) -- (10,6);}
{\draw (10,6) -- (8,6);}
{\draw (8,6) -- (8,4);}
{\node[above right=0pt] at (6,4.4) { $1,2$};}
{\draw (6,4) -- (8,4);}
{\draw (8,4) -- (8,6);}
{\draw (8,6) -- (6,6);}
{\draw (6,6) -- (6,4);}
{\node[above right=0pt] at (4,4.4) { $2,2$};}
{\draw (4,4) -- (6,4);}
{\draw (6,4) -- (6,6);}
{\draw (6,6) -- (4,6);}
{\draw (4,6) -- (4,4);}
{\node[above right=0pt] at (12,6.4) { $0,3$};}
{\draw (12,6) -- (14,6);}
{\draw (14,6) -- (14,8);}
{\draw (14,8) -- (12,8);}
{\draw (12,8) -- (12,6);}
{\node[above right=0pt] at (10,6.4) { $1,3$};}
{\draw (10,6) -- (12,6);}
{\draw (12,6) -- (12,8);}
{\draw (12,8) -- (10,8);}
{\draw (10,8) -- (10,6);}
{\node[above right=0pt] at (8,6.4) { $2,3$};}
{\draw (8,6) -- (10,6);}
{\draw (10,6) -- (10,8);}
{\draw (10,8) -- (8,8);}
{\draw (8,8) -- (8,6);}
{\node[above right=0pt] at (6,6.4) { $3,3$};}
{\draw (6,6) -- (8,6);}
{\draw (8,6) -- (8,8);}
{\draw (8,8) -- (6,8);}
{\draw (6,8) -- (6,6);}
{\node[above right=0pt] at (-5.0,0) {\normalsize $L_{w} \otimes $};}
{\node[above right=0pt] at (-3.0,2) {\normalsize $L_{w + 1} \otimes $};}
{\node[above right=0pt] at (-1.0,4) {\normalsize $L_{w + 2} \otimes $};}
{\node[above right=0pt] at (1.0,6) {\normalsize $L_{w + 3} \otimes $};}
{\node[above right=0pt] at (8,8) {\large $\iddots$};}
{\node[above right=0pt] at (10,8) {\large $\vdots$};}
{\node[above right=0pt] at (12,8) {\large $\vdots$};}
{\node[above right=0pt] at (15.0,15.0) {\large $\iddots$};}
{\node[above right=0pt] at (18.0,15.0) {\large $\vdots$};}
{\node[above right=0pt] at (20.0,15.0) {\large $\vdots$};}
{\node[above right=0pt] at (22.0,15.0) {\large $\vdots$};}
{\node[above right=0pt] at (24.0,15.0) {\large $\vdots$};}
{\node[above right=0pt] at (9.8,10.6) { $q,$};}
{\node[above right=0pt] at (9.8,10) { $q$};}
{\draw (10,10) -- (12,10);}
{\draw (12,10) -- (12,12);}
{\draw (12,12) -- (10,12);}
{\draw (10,12) -- (10,10);}
{\node[above right=0pt] at (11.8,10.6) { $q - 1,$};}
{\node[above right=0pt] at (11.8,10) { $q$};}
{\draw (12,10) -- (14,10);}
{\draw (14,10) -- (14,12);}
{\draw (14,12) -- (12,12);}
{\draw (12,12) -- (12,10);}
{\node[above right=0pt] at (11.8,12.6) { $q + 1,$};}
{\node[above right=0pt] at (11.8,12) { $q + 1$};}
{\draw (12,12) -- (14,12);}
{\draw (14,12) -- (14,14);}
{\draw (14,14) -- (12,14);}
{\draw (12,14) -- (12,12);}
{\node[above right=0pt] at (5.0,10) {\normalsize $L_{w+q} \otimes$};}
{\node[above right=0pt] at (7.0,12) {\normalsize $L_{w+q + 1} \otimes$};}
{\node[above right=0pt] at (16,10) {\large $\cdots$};}
{\node[above right=0pt] at (16,12) {\large $\cdots$};}
{\node[above right=0pt] at (19.8,10.6) { $0,$};}
{\node[above right=0pt] at (19.8,10) { $q$};}
{\draw (20,10) -- (22,10);}
{\draw (22,10) -- (22,12);}
{\draw (22,12) -- (20,12);}
{\draw (20,12) -- (20,10);}
{\node[above right=0pt] at (17.8,10.6) { $1,$};}
{\node[above right=0pt] at (17.8,10) { $q$};}
{\draw (18,10) -- (20,10);}
{\draw (20,10) -- (20,12);}
{\draw (20,12) -- (18,12);}
{\draw (18,12) -- (18,10);}
{\node[above right=0pt] at (23.8,12.6) { $0,$};}
{\node[above right=0pt] at (23.8,12) { $q + 1$};}
{\draw (24,12) -- (26,12);}
{\draw (26,12) -- (26,14);}
{\draw (26,14) -- (24,14);}
{\draw (24,14) -- (24,12);}
{\node[above right=0pt] at (21.8,12.6) { $1,$};}
{\node[above right=0pt] at (21.8,12) { $q + 1$};}
{\draw (22,12) -- (24,12);}
{\draw (24,12) -- (24,14);}
{\draw (24,14) -- (22,14);}
{\draw (22,14) -- (22,12);}
{\node[above right=0pt] at (19.8,12.6) { $2,$};}
{\node[above right=0pt] at (19.8,12) { $q + 1$};}
{\draw (20,12) -- (22,12);}
{\draw (22,12) -- (22,14);}
{\draw (22,14) -- (20,14);}
{\draw (20,14) -- (20,12);}
{\node[above right=0pt] at (17.8,12.6) { $3,$};}
{\node[above right=0pt] at (17.8,12) { $q + 1$};}
{\draw (18,12) -- (20,12);}
{\draw (20,12) -- (20,14);}
{\draw (20,14) -- (18,14);}
{\draw (18,14) -- (18,12);}
{\node[above right=0pt] at (19.8,20.6) { $2q,$};}
{\node[above right=0pt] at (19.8,20) { $2q$};}
{\draw (20,20) -- (22,20);}
{\draw (22,20) -- (22,22);}
{\draw (22,22) -- (20,22);}
{\draw (20,22) -- (20,20);}
{\node[above right=0pt] at (21.8,20.6) { $2q - 1,$};}
{\node[above right=0pt] at (21.8,20) { $2q$};}
{\draw (22,20) -- (24,20);}
{\draw (24,20) -- (24,22);}
{\draw (24,22) -- (22,22);}
{\draw (22,22) -- (22,20);}
{\node[above right=0pt] at (21.8,22.6) { $2q + 1,$};}
{\node[above right=0pt] at (21.8,22) { $2q + 1$};}
{\draw (22,22) -- (24,22);}
{\draw (24,22) -- (24,24);}
{\draw (24,24) -- (22,24);}
{\draw (22,24) -- (22,22);}
{\node[above right=0pt] at (23.8,24.6) { $2q + 2,$};}
{\node[above right=0pt] at (23.8,24) { $2q + 2$};}
{\draw (24,24) -- (26,24);}
{\draw (26,24) -- (26,26);}
{\draw (26,26) -- (24,26);}
{\draw (24,26) -- (24,24);}
{\node[above right=0pt] at (23.8,22.6) { $2q,$};}
{\node[above right=0pt] at (23.8,22) { $2q + 1$};}
{\draw (24,22) -- (26,22);}
{\draw (26,22) -- (26,24);}
{\draw (26,24) -- (24,24);}
{\draw (24,24) -- (24,22);}
{\node[above right=0pt] at (17.8,18.6) { $2q - 1,$};}
{\node[above right=0pt] at (17.8,18) { $2q - 1$};}
{\draw (18,18) -- (20,18);}
{\draw (20,18) -- (20,20);}
{\draw (20,20) -- (18,20);}
{\draw (18,20) -- (18,18);}
{\node[above right=0pt] at (19.8,18.6) { $2q - 2,$};}
{\node[above right=0pt] at (19.8,18) { $2q - 1$};}
{\draw (20,18) -- (22,18);}
{\draw (22,18) -- (22,20);}
{\draw (22,20) -- (20,20);}
{\draw (20,20) -- (20,18);}
{\node[above right=0pt] at (13.0,18) {\normalsize $L_{w+2q - 1} \otimes $};}
{\node[above right=0pt] at (15.0,20) {\normalsize $L_{w+2q} \otimes $};}
{\node[above right=0pt] at (17.0,22) {\normalsize $L_{w+2q + 1} \otimes $};}
{\node[above right=0pt] at (19.0,24) {\normalsize $L_{w+2q + 2} \otimes $};}
{\node[above right=0pt] at (0,-2) {\footnotesize $\pi_{2w, w}$};}
{\node[above right=0pt] at (2,-4) {\footnotesize $\pi_{2w - 1, w}$};}
{\node[above right=0pt] at (4,-2) {\footnotesize $\pi_{2w - 2, w}$};}
{\node[above right=0pt] at (6,-4) {\footnotesize $\pi_{2w - 3, w}$};}
{\node[above right=0pt] at (8,-2) {\footnotesize $\pi_{2w - 4, w}$};}
{\node[above right=0pt] at (10,-4) {\footnotesize $\pi_{2w - 5, w}$};}
{\node[above right=0pt] at (12,-2) {\footnotesize $\pi_{2w - 6, w}$};}
{\node[above right=0pt] at (16,-3.0) {\large $\cdots$};}
{\node[above right=0pt] at (18,-4) {\footnotesize $\pi_{2w-2q - 1, w}$};}
{\node[above right=0pt] at (20,-2) {\footnotesize $\pi_{2w-2q, w}$};}
{\node[above right=0pt] at (22,-4) {\footnotesize $\pi_{2w-2q + 1, w}$};}
{\node[above right=0pt] at (24,-2) {\footnotesize $\pi_{2w-2q + 2, w}$};}
{\draw[->] (-5.0,-0.5) -- (-5.0,26.5);}
{\draw[->] (-5.0,-0.5) -- (29.0,-0.5);}
{\node[above right=0pt] at (-5.0,25.5) { $q$};}
{\node[above right=0pt] at (28.0,-0.5) { $p$};}
\end{tikzpicture}

  \caption{The $E^1$-page of the slice spectral sequence of $\MGL$ with Adams grading in weight $w$.
    Each box with indices $(p, q)$ represents a copy of $H^{p,q}(F;\Z)$ tensored with $L_{w + q}$.
    The $d^r$-differentials goes one step to the left and $r$-steps upwards.
    The filtration degree is along the vertical $q$-axis, so the abutment is read off from the $E^\infty$-page vertically.
    The motivic homotopy group each column contributes to is indicated below the $E^1$-page.
  }
  \label{fig:E1-page}
\end{figure}
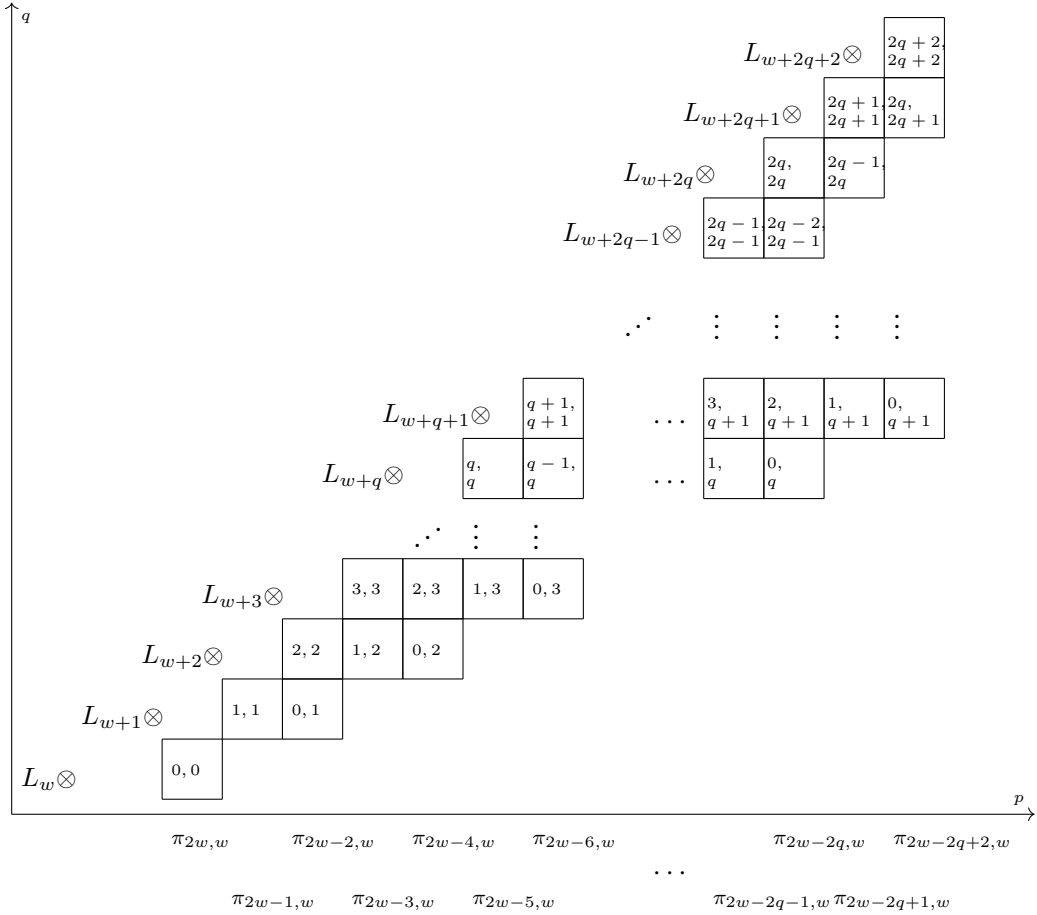

\section{Algebraic cobordism of the real numbers}
\label{sec:MGLR}
In this section we compute $\MGL_{*,*}(\R;\Z_2)$.
We first consider the slice spectral sequence which converges to $\MGL_{*,*}(\R;\Z/2^n)$, $n > 1$, and calculate its $E^\infty$-page.
Then we pass to the limit to obtain the homotopy groups of the $2$-complete algebraic cobordism spectrum $\MGL_{2}^{\wedge}$. Throughout this section we always assume $n > 1$.
\subsection*{Mod $2^n$ calculations}
Since $\MGL/2^n$ is a ring spectrum for $n > 1$, cf.~\cite{Oka},
the slice spectral sequence is multiplicative,
and we have a Leibniz rule. %
Hence, it suffices to determine the differentials on the algebra generators.
As an algebra, $E^1_{*,*,*}$ is generated by
$E^1_{-1,0,-1} = H^{1,1}\tensor L_0$,
$E^1_{0,0,-1}  = H^{0,1}\tensor L_0$,
$E^1_{0,0,-2}  = H^{0,2}\tensor L_0$
and
$E^1_{2i,i,i}  = H^{0,0}\tensor L_i$, for $i \in \Z$.
The algebra generators are $\rho$, $\tau$, $u$ and $x_i$ in degree $(-1, 0, -1)$, $(0, 0, -1)$, $(0, 0, -2)$ and $(2i, i, i)$.
For degree reasons, the only algebra generators on the $E^r$-page which can support differentials are powers of $u$ in $E^r_{0,0,2*} \cong H^{0,2*}\tensor L_0$.
Their differentials are described in \Cref{thm:diffs} below.

We need some results of \cite{HHR} for the group $G = C_2$.
In this case the norm map $N = N^{C_2}_{C_2}$ is the identity \cite[Proposition 2.27]{HHR},
and the norm construction on $\MU$ is simply $\MU$, that is $\MU^{((C_2))} = \MU$ \cite[5.1]{HHR}.
We have a factorization $L_* \to \pi_{2*,*}(\MGL) \to \pi_{*(1+\sigma)}^{C_2}(\MU)$,
and choose algebra generators $x_i$ of $L_*$ which map to the $\bar{r}_i \in \pi_{i + i \sigma}^{C_2}(\MU)$ defined in \cite[5.4.2]{HHR}. %
We let the image of the $x_i \in \pi_{2i,i}(\MGL)$ (also denoted $x_i$) be the canonical choice of algebra generators for $\pi_{2*,*}(\MGL)$.
The geometric fixed points of $\MU$ is the unoriented cobordism spectrum $\MO$.
\begin{proposition}[\protect{\cite[Proposition 5.50]{HHR}}]
\label{lem:HHR1}
The canonical map
\[
\pi_{*(1 + \sigma)}(\MU) \to \pi_*(\Phi^{C_2}(\MU)) \cong \pi_*(\MO)
\]
maps $x_i$ to zero if and only if $i = 2^k - 1$ for some $k$.
\end{proposition}
\begin{proposition}[\protect{\cite[Proposition 7.6]{HHR}}]
\label{lem:HHR2}
The canonical map
\[
\pi_*(\MO) \to \pi_*(\Phi^{C_2}(\HZ_{(2)}))
\]
is zero for $* > 0$.
\end{proposition}

The next lemma limits the possible targets of the differentials.
\begin{lemma}
\label{lem:rhotors}
The canonical map
\[
\pi_{*,*}(\MGL) \to \pi_{*,*}(\MGL[\rho^{-1}]) \cong \pi_*(\MO)
\]
sends
$x_i \in \pi_{2i,i}(\MGL)$ to $0$ if and only if $i = 2^k - 1$, for some $k$.
In particular, only the elements $x_{2^k - 1}$ are $\rho$-torsion.
\end{lemma}
\begin{proof}
Consider the diagram
\[
\begin{tikzcd}
\pi_{*,*}(\MGL) \ar[r]\ar[d] & \pi_{*,*}(\MGL[\rho^{-1}]) \ar[d, "\cong"] \\
\pi_{*,*}^{C_2}(\MU) \ar[r] & \pi_{*,*}^{C_2}(\MU)[a^{-1}] \cong \pi_{*}(\MO)
\end{tikzcd}
\]
where $a$ is the Euler class.
The left vertical map sends $x_i$ to $x_i$.
The bottom map sends $x_i$ to $0$ if and only if $i = 2^k - 1$ by \Cref{lem:HHR1}.
The right vertical map is an isomorphism by \Cref{thm:tom}.
Hence the assertion is proven
(unoriented cobordism has homotopy groups $\pi_*(\MO) = \Z/2[h_j, j \neq 2^{k} - 1]$, where the $h_j$'s are classes in degree $j$ defined as the coefficients of a certain power series in the Euler class $a$ given in \cite[Lemma 5.21]{HHR}).
\end{proof}

\begin{theorem}
\label{thm:diffs}
The $d^i$-differential in the slice spectral sequence of $\MGL/2^n$ on
$u^{2^{k-1}}$
is trivial for $i < r = 2^k - 1$ and
$
d^r(u^{2^{k-1}}) = \rho^{2^{k+1} - 1} x_{2^k - 1}.
$
\end{theorem}
\begin{proof}
We use the proof in \cite[Theorem 9.9]{HHR} adapted to the slice spectral sequence of $\MGL/2$.

The projection $\MGL/2^n \to \MGL/2$ induces a surjection on the slices $H^{0,q}(\R;\Z/2^n) \to H^{0,q}(\R;\Z/2)$ when is $q$ even.
Let $k \geq 1$ and do induction on $k$.
The differentials of $u^{2^{k-1}}$ in $\MGL/2^n$ determines the differentials on $\tau^{2^k}$ in $\MGL/2$.
\Cref{lem:rhotors} and induction implies that the only possible nonzero differential on $\tau^{2^k}$ in $\MGL/2$ is
\[
d^r(\tau^{2^k}) = \rho^{2^{k+1}-1} x_{2^k - 1}.
\]

We must show that the $d^r$-differential on $\tau^{2^k}$ is nonzero, i.e., $\tau^{2^k}$ is not a permanent cycle.
Assume $\tau^{2^k}$ is a permanent cycle.
That is, it represents some element in $\pi_{0,-2^k}(\MGL/2)$.
Consider the commutative diagram
\[
\begin{tikzcd}
\MGL/2 \ar[r]\ar[d] & \MZ/2 \ar[d] \\
\MGL/2[\rho^{-1}] \ar[r] & \MZ/2[\rho^{-1}],
\end{tikzcd}
\]
and the induced map of slice spectral sequences.
On $E^1$-pages, the $\tau^{2^k}$ in $E^1(\MGL/2)$ maps to $\tau^{2^k}$ in $E^1(\MZ/2)$.
The slice spectral sequence of $\MZ/2$ collapses on the $E^1$-page.
If $\tau^{2^k}$ survives the slice spectral sequence of $\MGL/2$, it represents a nonzero element $[\tau^{2^k}]$ in $\pi_{*,*}(\MGL/2)$ which maps to $\tau^{2^k}$ in $\pi_{*,*}(\MZ/2)$.
Note that $\tau^{2^k}$ is not $\rho$-torsion in $\pi_{*,*}(\MZ/2)$,
hence $[\tau^{2^k}]$ is not $\rho$-torsion, and thus survives $\rho$-localization.
That is,
$[\tau^{2^k}] \in \pi_{0,-2^k}(\MGL/2[\rho^{-1}])$ is nonzero and maps to
$\tau^{2^k} \in \pi_{0,-2^k}(\MZ/2[\rho^{-1}])$.
Identifying $\MGL/2[\rho^{-1}]$ with $\Phi^{C_2}(\RC(\MGL/2)) = \MO/2$,
and $\MZ[\rho^{-1}]$ with $\Phi^{C_2}(\RC(\MZ/2)) = \Phi^{C_2}(\HZ/2)$ we conclude that
the map $\pi_*(\MO) \to \pi_*(\Phi^{C_2}(\HZ_{(2)}))$ is nonzero for $* > 0$,
contradicting \Cref{lem:HHR2}.
Indeed, $\pi_*(\MO/2)$ is an extension of
$\Tor^\Z_1(\pi_*(\MO), \Z/2)$
by $\pi_*(\MO)/2$, and similarly for $\pi_*(\Phi^{C_2}(\HZ/2))$.
Hence, in the induced map of extensions we get that
$\pi_*(\MO) \to \pi_*(\Phi^{C_2}(\HZ_{(2)}))$ is nonzero for $* > 0$, since $k \geq 1$.
\end{proof}
The differentials in \Cref{thm:diffs} together with the Leibniz rule determine all possible differentials.
We proceed to describe the $E^\infty$-page of the slice spectral sequence of $\MGL/2^n$ over $\R$.
Consider the subalgebra
\[
B_{*,*} = \bigoplus_{l,p\geq 0}E^1_{2l-p,l,l-p}
\]
of the $E^1$-page
(i.e., $B_{*,*}$ is everything on the $E^1$-page not a multiple of $\tau$ or $u$, it is suggestive to write $B_{*,*} = H^{*=*}\tensor L_*$).
Above we observed
that $d^r(B_{*,*}) = 0$ for degree reasons for all $r$.
Hence, each $E^r$-page is a $B_{*,*}$-module.

The $E^1$-page is the $B_{*,*}$-module
\begin{equation}
\label{eq:E1}
\bigoplus_{i\geq0} u^i B_{*,*}
\oplus
\bigoplus_{i\geq0} \tau u^i B_{*,*}.
\end{equation}
\begin{theorem}
\label{thm:E8R}
Let $I_l$ be the ideal of $B_{*,*}$ generated by
\[
(\rho^3 x_1, \dots, \rho^{2^{l+1}-1}x_{2^l-1}),
\]
and let $I_0 = (0)$.
Let $I$ be $I_\infty$.
Let $J_l$ be the ideal $(2, x_1, \dots, x_{2^l-1})$ in $B_{*,*}$,
and let $J_0 = (0)$.
Let $A(l)$ denote the quotient $J_l/I_l$.
Then the $E^\infty$-page is
\begin{align}
E^\infty(\R;\Z/2^n)
\cong &
B_{*,*}/I \oplus \tau B_{*,*}/I \nonumber\\
&\oplus \bigoplus_{i \geq 1} u^i A(\nu_{2}(i)) L_*
\oplus \bigoplus_{i \geq 1} \tau u^i A(\nu_{2}(i)) L_*. \label{eq:Einfty}
\end{align}
Here $L_*$ is inserted into $E^\infty(\R;\Z/2^n)$ via the canonical embedding
induced by $L_k \cong E^1_{2k,k,k}$.
\end{theorem}
\begin{proof}
Let $l_2(r)$ be the length of the binary expansion of $r$, i.e., the minimal integer $l_2(r)$ such that $r < 2^{l_2(r)}$.
By induction on $r$ we get
\begin{align}
E^r(\R;\Z/2^n) =
&\bigoplus_{\nu_2(i)\geq l_2(r)-1} u^i B_{*,*}/I_{l_2(r)-1}
\oplus 
\bigoplus_{\nu_2(i)\geq l_2(r)-1} \tau u^i B_{*,*}/I_{l_2(r)-1} \nonumber \\
&\oplus \bigoplus_{l_2(r)-1 > i \geq 1} u^i A(\nu_{2}(i)) L_*
\oplus \bigoplus_{l_2(r)-1 > i \geq 1} \tau u^i A(\nu_{2}(i)) L_*.
\label{eq:Er}
\end{align}
Indeed, the base case is \eqref{eq:E1}.
Assume inductively \eqref{eq:Er} is true for $r$.
If $l_2(r) = l_2(r+1)$ then $E^r = E^{r+1}$,
since there are only differentials when 
$l_2(r) < l_2(r+1)$.
So we may assume $r=2^k-1$.
Then we have a $d^{r}$-differential
\[
d^{r}(u^{2^{k-1}})
=
\rho^{2^{k+1} -1}x_{2^{k}-1},
\]
and more generally $d^{r}$-differentials on the $u^{i}$ with $2$-adic valuation $\nu_2(i) = k-1$.

We analyze the differential on each summand in \eqref{eq:Er}:
Note that the summands do not interact since the differentials are $u^{2^k}$- and $\tau$-linear.
The differential on the last two summands are zero for degree reasons, cf.~\Cref{fig:E1-page}.

The contribution to the $E^{r+1}$-page of the differential on the summand
\[
\bigoplus_{\nu_2(i) = l_2(r)-1} u^i B_{*,*}/I_{l_2(r)-1}
\]
is
\[
\bigoplus_{\nu_2(i)=l_2(r)-1} u^i A(\nu_{2}(i)) L_*.
\]
Indeed, the differential of an element
$a = u^i \rho^j y \in u^i B_{*,*}/I_{l_2(r)-1}$ is
\[
d^{r}(a)
=
\rho^{2^{k+1} -1 + j}x_{2^{k}-1} u^{i - 2^{k-1}} y.
\]
This is zero in the image if and only if
$\rho^{2^{k+1} -1 + j}y \in I_{l_2(r)-1}$.
That is, $\rho^j y$ is an element of $J_{l_2(r)-1}$
(note that $\rho^{2^{k+1} -1} y \equiv 0 \bmod I_{l_2(r)-1}$ for $y \in J_{l_2(r)-1}$).

Similarly the summand
\[
\bigoplus_{\nu_2(i) = l_2(r)-1} \tau u^i B_{*,*}/I_{l_2(r)-1}
\]
contributes
\[
\bigoplus_{\nu_2(i)=l_2(r)-1} \tau u^i A(\nu_{2}(i)) L_*.
\]
The Leibniz rule implies that summands with $\nu_2(i) > l_2(r)-1$ have zero $d^r$-differential.

Passing to the limit over $r$ in \eqref{eq:Er} we obtain \eqref{eq:Einfty}.
\end{proof}

\begin{remark}
Determining possible extensions and the multiplicative structure in $\MGL_{*,*}(\R; \Z/2^n)$ seems to be fairly hard.
The simplest extensions $2[\tau] = [\rho^2 x_1]$ and $2^{n-1}[2u] \neq 0$ can be determined by comparison with $K$-theory of the real numbers.
There are extension problems for $x_1 u, x_3 u, \dots, 2u, 2u^2, \dots$ and their $\tau$-multiples.
Unfortunately, $K$-theory does not detect the ones not a multiple of $x_1$.
For motivic Morava $K$-theory of the real numbers we are generally unable to determine the extensions, see \Cref{rmk:Kn-exts}.
\end{remark}

\subsection*{Two-complete algebraic cobordism of $\R$}%
\label{sec:R2}
Consider the filtration of $\MGL_{*,*}(\R;\Z/2^n)$ associated to $E^\infty(\R;\Z/2^n)$.
Since all the groups in the filtration are finite $\lim_n$ is exact and we may pass to the limit over $n$ to obtain a filtration of $\MGL_{*,*}(\R;\Z_2)$,
the motivic homotopy groups of $\holim_n \MGL/2^n$.
We write $E^\infty(\R;\Z_2)$ for the associated graded of this filtration.
This associated graded is simpler than $E^\infty(\R;\Z/2^n)$ since all the $\tau$-multiples disappear.
\begin{lemma}
The associated graded of the filtration of $\MGL_\star(\R;\Z_2)$ described above is
\begin{align}
E^\infty(\R;\Z_2)
\cong &
\lim_n \left( B_{*,*}/I \oplus \bigoplus_{i \geq 1} u^i A(\nu_{2}(i)) L_* \right) \nonumber \\
\cong &
\frac{\Z_2[\rho, z_{j,l}, x_h \vert l, j \geq 0, 2^{j}-1 \neq h \geq 1]}
{(2\rho, \rho^{2^{j+1} - 1}z_{j,l},
z_{k,l}z_{a,b} = z_{k,l+b2^{a-k}}z_{a,0}, \text{ for } a \geq k)}.
\end{align}
Here we write $z_{j,l} = u^{l2^{j}}x_{2^{j}-1}$,
which explains the relations (with the convention $x_0 = 2$).
\end{lemma}
\begin{proof}
This reduces to the computation in \Cref{lem:HR}:
The limit is induced by the system \eqref{eq:2n-sys},
and the structure maps in the inverse system are the same as in motivic cohomology. The multiples of $\tau$ all have multiplication by $2$ as structure maps.
Since these are all $2$-torsion they vanish in the limit.
Hence we are left with computing the limit of
\begin{align}
B_{*,*}/I \oplus \bigoplus_{i \geq 1} u^i A(\nu_{2}(i)) L_*. \nonumber
\end{align}
Anything not a multiple of $\rho$ is part of a system which is $\Z_2$ in the limit.
Otherwise the structure map is the identity of $\Z/2$ and in the limit we get $\Z/2$.
\end{proof}
\begin{theorem}
\label{thm:MGLR}
There are no hidden additive or multiplicative extensions on the $E^\infty$-page.
That is, as an algebra
\[
\bigoplus_{p,w} \bigoplus_q E^\infty_{p,q,w} \cong \MGL_{*,*}(\R;\Z_2).
\]
Hence,
\begin{equation}
\MGL_{*,*}(\R;\Z_2)
=
\frac{\Z_2[\rho, z_{j,l}, x_h \vert l, j \geq 0, 2^{j}-1 \neq h \geq 1]}
{(2\rho, \rho^{2^{j+1} - 1}z_{j,l}, z_{k,l}z_{a,b} = z_{k,l+b2^{a-k}}z_{a,0}, \text{ for } a \geq k)}.
\end{equation}
Here $z_{j,l} = [u^{l2^{j}}x_{2^{j}-1}]$ (with the convention $x_0 = 2$).
That is, $z_{j,l}$ is a particular representative of $u^{l2^{j}}x_{2^{j}-1}$
(during the proof we will see that it is a particular representative).
\end{theorem}
\begin{proof}
Note that any element on the $E^1$-page is a sum of monomials $u^i\rho^jx_K$ for some numbers $i, j$ and multi-index $K$. The degrees of $u$, $\rho$ and $x_K$ are linearly independent.
It is helpful to keep in mind that a specific element of $\MGL_{*,*}(\R;\Z_2)$ is obtained by going downwards along $u$, $(2,1)$-diagonally upwards along $x_K$ and $(-1,-1)$-diagonally downwards along $\rho$. We will perform induction on the power of $u$.
  
We choose representatives $z_{j,l} = [u^{l2^{j}}x_{2^{j}-1}]$ in $\MGL_{*,*}(\R;\Z_2)$ of the algebra generators $u^{l2^{j}}x_{2^{j}-1}$ on the $E^\infty$-page inductively such that $\rho^i[u^{l2^{j}}x_{2^{j}-1}] = [\rho^i u^{l2^{j}}x_{2^{j}-1}]$, in particular $\rho^{2^{l+1} -1}z_{l,j} = 0$ (here $[\rho^i u^{l2^{j}}x_{2^{j}-1}]$ is some element representing $\rho^i u^{l2^{j}}x_{2^{j}-1}$).
Indeed, in general
\[
\rho^i[u^{l2^{j}}x_{2^{j}-1}] = [\rho^i u^{l2^{j}}x_{2^{j}-1}] + [\rho^{k} u^{l'2^{j'}}x_K] + \dots,
\]
where $K > 2^j - 1$ (so it lies in a higher filtration), which implies $k > i$ and $l2^j > l'2^{j'}$ (hence the second term is already defined by induction on $l$).
Hence, we could as well have chosen
\[
[u^{l2^{j}}x_{2^{j}-1}] - [\rho^{k-i} u^{l'2^{j'}}x_K]
\]
as a representative for $u^{l2^{j}}x_{2^{j}-1}$. Inductively we get
\[
\rho^i[u^{l2^{j}}x_{2^{j}-1}] = [\rho^i u^{l2^{j}}x_{2^{j}-1}]
\]
(this is really only a condition for $i = 2^{j+1} - 1$).

It remains to show
$z_{k,l}z_{a,b} = z_{k,l+b2^{a-k}}z_{a,0}$, $a \geq k$.
For the remainder of the proof we use the notation $z_{k,l}= [u^{l2^{k}}x_{2^{k}-1}]$.
In general,
\[
[u^{l2^{k}}x_{2^{k}-1}][u^{b2^{a}}x_{2^{a}-1}]
= [u^{(l + b2^{a - k})2^{k}}x_{2^{k}-1}][x_{2^{a}-1}] + \rho^i[u^{n2^m}x_{2^m-1}][x_K] + \dots.
\]
Without loss of generality we may assume $a \geq k$, and that there is no $m' < m$ such that $x_{2^{m'} - 1}$ divides $x_K$.
Note that if we multiply the left hand side by $\rho^{2^{k+1}-1}$ we get 0.
For the right hand side to be zero we must have $2^{k+1} -1 + i \geq 2^{m+1} - 1$.
Furthermore, for $\rho^i[u^{n2^m}x_{2^m-1}]x_K$ to be nonzero we must have $i < 2^{m+1}-1$.
Both sides are in the same bidegree of $\MGL_{*,*}(\R;\Z_2)$, and the term $\rho^i[u^{n2^m}x_{2^m-1}]x_K$ is in a higher filtration.
This gives a system of equations and inequalities on $k,l,a,b,m,n,K,i$ (by abuse of notation $K$ denotes either a multi-index or the degree of the multi-index).
This system has no solutions. We now carry this out in detail.

Comparison of the left and right hand side gives the following relations
\[
\begin{matrix}
2(2^k - 1) + 2(2^a - 1) \\
2^k - 1 + 2^a - 1 \\
(1 - 2l)2^k - 1 + (1 - 2b)2^a - 1
\end{matrix}
\begin{matrix}
= \\
\leq \\
=
\end{matrix}
\begin{matrix}
2(2^m - 1) + 2K - i \\
2^m - 1 + K \\
(1 - 2n)2^m - 1 + K - i
\end{matrix}
\]
Additionally we have
\[
a \geq k,\qquad 2^{k+1} + i \geq 2^{m+1},\qquad 2^{m+1}-1 > i > 0.
\]
Solving for $i/2$ ($i$ is a multiple of 4) this gives
\[
2l2^k + 2b2^a - 2n2^m = (2l + 2b2^{a-k})2^k - 2n2^m  = i/2
\]
which must satisfy
\[
2^k + 2^{m} > 2^k + 2l2^k + 2b2^a - 2n2^m \geq 2^m
\]
and $i/2 > 0$.
Equivalently,
\begin{equation}
\label{eq:linineq}
2^{m} > 2l2^k + 2b2^a - 2n2^m = (l + b2^{a-k})2^{k+1} - n2^{m+1} \geq \max(2^m - 2^k,1).
\end{equation}
If $k \geq m$ there are clearly no solutions.
If $k < m$ this is equivalent to
$(1 + 2n)2^{m} >  (l + b2^{a-k})2^{k+1} \geq (1 + 2m)2^m - 2^k$,
i.e., $(l + b2^{a-k})2^{k+1}$ is in the half open interval $[(1 + 2n)2^m - 2^k, (1 + 2n)2^{m})$. But the endpoint $(1 + 2n)2^{m} = (1 + 2n)2^{m-k-1}2^{k+1}$ is a multiple of $2^{k+1}$, i.e., the interval never contains $(l + b2^{a-k})2^{k+1}$.
\end{proof}
\begin{remark}
\label{rmk:elementary}
As an alternative to the above proof it is possible to copy the proof of \cite[Theorem 4.11]{Hu-Kriz:real} (or rather \cite[Theorem 7.4]{Hu-Kriz:real}) verbatim.
That is, use the motivic Adams spectral sequence to determine the multiplicative extensions (in fact, the computation of Hu and Kriz takes place in an $\Ext$-group of a subalgebra of the $C_2$-equivariant Steenrod algebra. This subalgebra is isomorphic to the mod 2 motivic Steenrod algebra).
A second alternative is to take $C_2$-equivariant complex realization, and compare directly with the $C_2$-equivariant homotopy groups of $\BPR$.
Conversely the above proof applies to the setting of Hu and Kriz, since \eqref{eq:linineq} has no solutions even if $b, l, n$ are allowed to take negative values.
This provides an alternative route to the $C_2$-equivariant homotopy groups  of $\BPR$ without use of the Adams spectral sequence.
\end{remark}

\section{Algebraic cobordism of real number fields}
\label{sec:MGLF}
In this section we compare the slice spectral sequence of $\MGL$ over a number field $F$ with the slice spectral sequence over $\R$. This determines all the differentials in the slice spectral sequence over $F$ and we compute the associated graded of $\MGL_{*,*}(F;\Z)$.

Consider the map of $E^r$-pages of the slice spectral sequence for $\MGL$ induced by the real embeddings of $F$
\[
f^r_{p,q,w} : E^r_{p,q,w}(F;\Z) \to \oplus^{r_1} E^r_{p,q,w}(\R;\Z_2).
\]
This gives a commutative diagram
\begin{equation}
\label{eq:diff-diag}
\begin{tikzcd}[column sep=70pt]
E^r_{p,q,w}(F;\Z) \ar[r, "f^r_{p,q,w}"]\ar[d, "d^r(F)"] & \oplus^{r_1} E^r_{p,q,w}(\R;\Z_2)\ar[d, "\oplus^{r_1}d^r(\R)"] \\
E^r_{p-1,q+r,w}(F;\Z) \ar[r, "f^r_{p-1,q+r,w}"] & \oplus^{r_1} E^r_{p-1,q+r,w}(\R;\Z_2).
\end{tikzcd}
\end{equation}
Whenever the source of $d^r(F)$ is nonzero, the target $E^r_{p-1,q+r,w}(F;\Z)$ is isomorphic to $(\Z/2)^{\oplus k}$ for some $k$ and $f^r_{p-1,q+r,w}$ is an isomorphism.
That is, the differential $d^r(F)$ is determined by the differential $d^r(\R)$ and the map $f^r_{p,q,w}$, i.e., of the sum of the real embeddings.
By \Cref{lem:HFR} and \Cref{lem:H3R} the map $f^r_{p,q,w}$ is surjective except possibly when restricted to a summand $H^{0,q'}(F;\Z)$ or $H^{1,q'}(F;\Z)$ of $E^r_{p,q,w}(F;\Z)$.
We define
\[
\ol{H}^{p,q,(l)}(F;\Z) = \coker(
g: H^{p-(2^{l+1}-1),q-(2^l-1)}(F;\Z) \to  H^{p,q}(F;\Z)
)
\]
where $g$ is the unique arrow making the diagram 
\begin{equation}
\label{eq:H-bar}
\begin{tikzcd}[column sep=70pt]%
H^{p-(2^{l+1}-1),q-(2^l-1)}(F;\Z) \ar[d]\ar[r, dashed, "g"] & H^{p,q}(F;\Z)\ar[d]  \\
\oplus^{r_1}H^{p-(2^{l+1}-1),q-(2^l-1)}(\R;\Z_2) \ar[r, "\rho^{2^{l+1}-1}u^{-2^{l-1}}"] & \oplus^{r_1} H^{p,q}(\R;\Z_2)
\end{tikzcd}
\end{equation}
commute (note that either $H^{p-(2^{l+1}-1),q-(2^l-1)}(F;\Z) = 0$, or the right vertical map in \eqref{eq:H-bar} is an isomorphism).
We define
\[
\wt{H}^{p,q}(F;\Z) = \ker(H^{p,q}(F;\Z) \to \oplus^{r_1}H^{p,q}(\R;\Z_2)).
\]
Note that
$\ol{H}^{p,q,(l)}(F;\Z) = 0$ if $p-(2^{l+1}- 1) > 1$
and $\wt{H}^{p,q}(F;\Z) = 0$ if $p \geq 3$.

\begin{theorem}
\label{thm:MGLF}
The $E^\infty$-page of the slice spectral sequence is
\[
E^\infty(F;\Z) \cong \bigoplus_{p,q,K} A^{p,q,K}
\]
where $K$ runs over all multi-indices of monomials in $L_*$.
Here
\begin{equation}
\label{eq:Apqk}
A^{p,q,K} = \begin{cases}
  \ol{H}^{p,q,(l)}(F;\Z) x_K & \text{if } l = \min\{l \colon x_{2^l-1} \vert x_K, 1 \leq l < \nu_2(q-p) \} < \infty \\
  \wt{H}^{p,q}(F;\Z) x_K & \text{if } x_{2^l - 1} \nmid x_K \text{ for all } 1 \leq l \leq \nu_2(q-p) \text{ and } 0 < \nu_2(q - p) < \infty \\
  H^{p,q}(F;\Z) x_K  & \text{otherwise}.
  \end{cases}
\end{equation}
By definition $\min \{ \} = \infty$.
\end{theorem}
\begin{proof}
We will prove inductively that
\begin{equation}
\label{eq:Erind}
E^r_{p,q,w}
\cong
\bigoplus_{p,q,K} A^{p,q,K,(r)}
\end{equation}
where
\begin{equation}
\label{eq:Apqk}
A^{p,q,K,(r)} = \begin{cases}
  \ol{H}^{p,q,(l)}(F;\Z) x_K & \text{if } l = \min\{l \colon x_{2^l-1} \vert x_K, 1 \leq l < l_{p,q,r} \} < \infty \\
  \wt{H}^{p,q}(F;\Z) x_K & \text{if } x_{2^l - 1} \nmid x_K \text{ for all } 1 \leq l \leq  l_{p,q,r}, \text{ and } 0 < \nu_2(q - p) < l_2(r)\\
  H^{p,q}(F;\Z) x_K  & \text{otherwise,}
  \end{cases}
\end{equation}
and $l_{p,q,r} = \min\{l_2(r)-1, \nu_2(q-p)\}$.

When $r = 1$ we are in the last case of \eqref{eq:Apqk}, so \eqref{eq:Erind} holds by definition.
Assume \eqref{eq:Erind} to be true inductively for $r$.
We must show that \eqref{eq:Erind} is true for $r+1$.
From \eqref{eq:diff-diag} we see that there are only $d^r$-differentials when $l_2(r) < l_2(r+1)$, so we may assume $r = 2^{k}-1$.
The only terms $A^{p,q,K,(r)}$ which can support a $d^r$ differentials are the terms with $\nu_2(q-p) = k$.
The target of a $d^r$-differential on $A^{p,q,K,(r)}$ is $A^{p',q',K,(r)}x_{2^k-1}$, for $q' = p + 2^{k+1}-1, q' = q+2^k - 1$.
We note that $\nu_2(q' - p') = \nu_2(q - p - 2^k) > k$.
Hence all terms with $\nu_2(q - p) < k$ are unchanged when passing from the $E^r$-page to the $E^{r+1}$-page.

Consider a term $A^{p,q,K}$ on the $E^r$-page.
Assume $\nu_{2}(q-p) = k$, and that $A^{p,q,K,(r)}$ is nonzero and supports a $d^r$-differential to $A^{p',q',K,(r)}x_{2^k-1}$.
Then $A^{p,q,K,(r)}$ is of the form $H^{p,q}(F;\Z)x_K$ or $\ol{H}^{p,q,(l)}(F;\Z)x_K$ for some $l < k$.
In the latter case the $d^r$-differential is zero, indeed, $\rho^{2^{k+1}-1}x_K = 0$.
In the former case the $d^r$-differential is nonzero if and only if
$\rho^{2^{k+1}-1}x_K$ is nonzero, that is, if and only if $x_{2^l - 1} \nmid x_K$ for all $1 \leq l < k$.
Hence $H^{p,q}(F;\Z)x_K$ is replaced by $\wt{H}^{p,q}(F;\Z)x_K$ if and only if $x_{2^l - 1} \nmid x_K$ for all $1 \leq l < k = \nu_{2}(q-p) = l_2(r+1)-1$.

Assume next $\nu_2(q-p) > k$.
That is, $A^{p,q,K,(r)}$ is potentially the target of a $d^r$-differential.
If we are in the first case of \eqref{eq:Apqk} there is nothing to show,
so assume that we are in the last case of \eqref{eq:Apqk}.
If $x_{2^k-1} \nmid x_K$ there can be no differential to $A^{p,q,K,(r)}$, so $A^{p,q,K,(r)} = A^{p,q,K,(r+1)} = H^{p,q}(F;\Z)x_K$.
If $x_{2^k-1} \vert x_K$ there is a differential from
$A^{p-(2^{k+1}-1),q - (2^{k}-1),K,(r)}/x_{2^k-1}$ and $A^{p,q,K,(r)}$ is replaced by $\ol{H}^{p,q,(k)}(F;\Z)x_K$.
This completes the inductive step.

Passing to the limit over $r$ we get the $E^\infty$-page.
\end{proof}
\begin{remark}
It is possible to give a presentation of $E^\infty(F;\Z)$ in terms of generators and relations.
There is a generator for every pair $(x_{2^k-1}, y), y \in \oplus_j H^{1,j2^k}(F;\Z)$,
and we must re-encode the multiplication in $H^{*,*}(F;\Z)$.
This can be compared with giving a presentation of
the subalgebra $k\{1, x^iy^j \vert i \geq 0, j > 1\} \subset k[x, y]$.
\end{remark}

\Cref{thm:MGLF} and comparison with $\MGL_\star(\R;\Z_2)$ determines $\MGL_\star(F;\Z)$ up to some indeterminacy in the additive and multiplicative structure in the part of $\MGL_\star(F;\Z)$ coming from $H^{1,*}(F;\Z)$ and $H^{2,*}(F;\Z)$.
For $2$-regular number fields with exactly one prime dividing $2$ (e.g., $\Q$) it might be possible to determine the extensions as in \cite{John-Paul}.

\section{Algebraic cobordism and $\zeta$-functions}
\label{sec:zeta}
Motivic cohomology of number fields and rings of $\mathcal S$-integers are qualitatively the same. In this section we exploit this likeness to compute $2$-complete algebraic cobordism over rings of $\mathcal S$-integers, $\mathcal S \supset \{2, \infty\}$ a finite set of places, and give a formula relating the order of the algebraic cobordism groups of rings of $2$-integers to special values of Dedekind $\zeta$-functions.
This is an analogy to Lichtenbaum's conjecture relating the order of algebraic $K$-theory of rings of integers to special values of Dedekind $\zeta$-functions, see \cite[Theorem 0.2]{Rognes-Weibel}.
An analogy to Lichtenbaum's conjecture for hermitian $K$-theory was proven in \cite[Theorem 1.10]{KRO}.
We use the motivic cohomology spectrum of Spitzweck over the ring of $\mathcal S$-integers in a number field $F$ \cite{Spitzweck:commutative}.

When $\mathcal{S} \supset \{2, \infty \}$ the residue fields of $\OFS$ all have odd characteristic so the characteristics are invertible in $\Z/2^n$.
Then the slices of $\MGL/2^n$ are 
$\s_q(\MGL/2^n) = \Sigma^{2q,q} \MZ/2^n \tensor L_q$ \cite[Theorem 11.3]{Spitzweck:commutative}.
When the set of places $\mathcal S$ is finite the 2-completed motivic cohomology groups are finitely generated outside of $(0,0)$, the Picard group $H^{2,1}(\OFS;\Z_2)$ can be nonzero, and we have an isomorphism $H^{p,q}(\OFS;\Z_2) \cong H^{p,q}(\R;\Z_2)$ for $p \geq 3$.
Hence the differentials in the slice spectral sequence over $\OFS$ are determined by the same procedure as over $F$ by comparison with the real embeddings.
For degree reasons the extra groups $H^{2,1}(\OFS;\Z_2)\tensor L_*$ do not interact with any terms in the slice spectral sequence or support any differentials.
They give rise to the extra summand $H^{2,1}(\OFS;\Z_2)\tensor L_{n+1}$ in $\MGL_{2n,n}(\OFS;\Z_2)$.
This summand has zero multiplication with motivic cohomology of positive weight.
The $2$-local motivic cohomology groups of $\OFS$ are finitely generated, hence $\lim_n$ is exact on the mod $2^n$ motivic cohomology groups, and completion before or after taking homotopy groups is the same for $\MZ$ and $\MGL$.
As in \Cref{sec:R2} we first run the slice spectral sequence for $\MGL/2^n$ over $\OFS$.
This gives filtrations of $\MGL_{*,*}(\OFS;\Z/2^n)$ for each $n$. Taking the limit over $n$ of the filtrations we obtain a filtration of $\MGL_{*,*}(\OFS;\Z_2)$.

\begin{theorem}
\label{thm:OFS}
\label{thm:MGLO}
The associated graded of the 2-complete algebraic cobordism groups of a ring of $\mathcal S$-integers in a number field $F$, $\mathcal S \supset \{2, \infty\}$, $\mathcal S$ finite, is
\[
E^\infty(\OFS;\Z_2) \cong \bigoplus_{p,q,K} A^{p,q,K}
\]
where $K$ runs over all multi-indices of monomials in $L_*$.
Here
\[
A^{p,q,K} = \begin{cases}
  \ol{H}^{p,q,(l)}(\OFS;\Z_2) x_K & \text{if } l = \min\{l \colon x_{2^l-1} \vert x_K, 1 \leq l < \nu_2(q-p) \} < \infty \\
  \wt{H}^{p,q}(\OFS;\Z_2) x_K & \text{if } x_{2^l - 1} \nmid x_K \text{ for all } 1 \leq l \leq \nu_2(q-p) \text{ and } 0 < \nu_2(q - p) < \infty \\
  H^{p,q}(\OFS;\Z_2) x_K  & \text{otherwise}.
  \end{cases}
\]
The groups $\ol{H}^{p,q,(l)}(\OFS;\Z_2)$ and $\wt{H}^{p,q}(\OFS;\Z_2)$ are defined similarly as in \Cref{sec:MGLF}.
\end{theorem}

Recall the following theorem of Manfred Kolster:
\begin{theorem}[\protect{\cite[Theorem A.1]{Rognes-Weibel}}]
\label{thm:Kolster}
Let $F$ be a totally real abelian number field and $k \geq 1$.
Then
\[
\zeta_{F}(1 - 2k) \sim_2 \frac{\# H^{2}_{\et}(\OO_F[\frac{1}{2}]; \Z_2(2k)) }{\# H^{1}_{\et}(\OO_F[\frac{1}{2}]; \Z_2(2k)) }.
\]
Here $a \sim_2 b$ means that $a$ and $b$ have the same $2$-adic valuation.
\end{theorem}
Define the subgroups
$L_t' = \Z\{x_K \in L_t \colon \exists n, x_{2^n - 1} \mid x_K \}$
and
$L_t'' = \Z\{x_K \in L_t \colon \forall n, x_{2^n - 1} \nmid x_K \}$
of $L_t$,
such that $L_t' \oplus L_t''= L_t$.
We write $n^{L_t}$ for $n^{\rk_{\Z}L_t}$ and so on.
Combining \Cref{thm:MGLO} with \Cref{thm:Kolster} we get the following relation between algebraic cobordism of $\OO_F[\frac{1}{2}]$ and special values of the Dedekind $\zeta$-function of $F$.
\begin{corollary}
\label{thm:zeta}
Let $F$ be a totally real abelian number field with $r_1$ real embeddings and $n \geq 1$ an integer.
Then
\begin{equation}
\label{eq:zeta}
2^{r_1L_{2n+w}''}
\frac{\# \MGL_{4n+2w-2,w}(\OO_F[\frac{1}{2}];\Z_2) }
{\# \MGL_{4n+2w-3,w-1}(\OO_F[\frac{1}{2}];\Z_2) }
\sim_2
(\zeta_{F}(1-2n))^{L_{2n+w}}.
\end{equation}
\end{corollary}
\begin{proof}
The cardinality of $\MGL_{p,w}(\OO_F[\frac{1}{2}];\Z_2)$ is the product of the cardinalities of the groups
\[
A^{-p + 2\vert K \vert, -w + \vert K \vert, K},\quad x_K \in L_*.
\]
The cardinality of $\MGL_{p-1,w-1}(\OO_F[\frac{1}{2}];\Z_2)$ is the product of the cardinalities of the groups
\[
A^{-p+1 + 2\vert K \vert, -w+1 + \vert K \vert, K},\quad x_K \in L_*.
\]
From now on we assume $p$ and $(p+2)/2 - w$ to be even.
Then multiplication by $\rho$ induces an isomorphism
\begin{equation}
\label{eq:rho-mult}
A^{-p + 2\vert K \vert, -w + \vert K \vert, K}
\to 
A^{-p +1 + 2\vert K \vert, -w + 1 + \vert K \vert, K}
\end{equation}
when $-p + 2\vert K \vert > 2$.
Indeed, this follows from parity considerations: %
Let $p' = -p + 2\vert K \vert, q' = -w + \vert K \vert$.
If $A^{p', q', K}$ is given by the two last cases of \eqref{eq:Apqk} then \eqref{eq:rho-mult} is clear,
so we may assume
\[
A^{p', q', K}
=
\ol{H}^{p',q',(l)}(F;\Z)x_K,
\qquad
A^{p'+1, q'+1, K}
=
\ol{H}^{p'+1,q'+1,(l)}(F;\Z)x_K,
\]
for some minimal $l$ such that $x_{2^{l}-1}\vert x_K$ and $q' - p' = j2^{l+1}$ for some $j$.
The statement is then equivalent to the map of cokernels
\[
\begin{tikzcd}
\coker\Bigl(H^{p'-(2^{l+1} - 1), q' - (2^l - 1)}(F;\Z) \ar[r]\ar[d, "\rho"] & \oplus^{r_1}H^{p'-(2^{l+1}-1),q'-(2^l - 1)}(\R;\Z_2)\Bigr)\ar[d, "\rho"] \\
\coker\Bigl(H^{p'+1-(2^{l+1} - 1), q'+1 - (2^l - 1)}(F;\Z) \ar[r] & \oplus^{r_1}H^{p'+1-(2^{l+1}-1),q'+1-(2^l - 1)}(\R;\Z_2)\Bigr)
\end{tikzcd}
\]
being an isomorphism.
In general this map is not an isomorphism if $p' - (2^{l+1} - 1) = 0, 1$,
but this possibility is excluded by the assumptions on $p$ and $w$.
Indeed, $p' - (2^{l+1} -1) \neq 0$,
since $p' = -p + 2 \vert K \vert$ is even by assumption.
If $p' - (2^{l+1} -1) = 1$,
then $-p + 2\vert K \vert = 2^{l+1}$ and $-w + \vert K \vert = (j+1)2^{l+1}$. Hence, $p - 2w = (j+1)2^{l+2} - 2^{l+1}$, contradicting $(p+2)/2 - w$ being even.

With the above assumptions on $p$ and $w$ we get
\begin{align*}
\frac{\# \MGL_{p,w}(\OO_F[\frac{1}{2}];\Z_2) }{\# \MGL_{p-1,w-1}(\OO_F[\frac{1}{2}];\Z_2) }
&=
\left(\frac{\# H^{2,(p+2)/2-w} }
{\# H^{1,(p+2)/2-w} }\right)^{L_{(p+2)/2}'}
\left(\frac{\# \wt{H}^{2,(p+2)/2-w} }
{\# \wt{H}^{1,(p+2)/2-w} }\right)^{L_{(p+2)/2}''} \\
&=
\left(\frac{\# H^{2,(p+2)/2-w} }
{\# H^{1,(p+2)/2-w} }\right)^{L_{(p+2)/2}'}
\left(2^{-r_1}\frac{\# H^{2,(p+2)/2-w} }
{\# H^{1,(p+2)/2-w} }\right)^{L_{(p+2)/2}''} \\
&=
2^{-r_1 L_{(p+2)/2}''} \left(\frac{\# H^{2,(p+2)/2-w} }
{\# H^{1,(p+2)/2-w} }\right)^{L_{(p+2)/2}}.
\end{align*}
Here we use \Cref{lem:H12R} for the second equality.
Hence,
\[
2^{r_1L_{(p+2)/2}''}\frac{\# \MGL_{p,w}(\OO_F[\frac{1}{2}];\Z_2) }{\# \MGL_{p-1,w-1}(\OO_F[\frac{1}{2}];\Z_2) }
\sim_2
(\zeta_F(1+w-(p+2)/2))^{L_{(p+2)/2}}.
\]
A change of variables yields \eqref{eq:zeta}.
\end{proof}

\section{Algebraic cobordism of totally imaginary number fields and fields of low $2$-cohomological dimension}
\label{sec:final}
For completeness we state the following results on algebraic cobordism of totally imaginary number fields and fields with $2$-cohomological dimension less than or equal to $2$.
For such fields the slice spectral sequence collapses on the $E^1$-page for degree reasons.
\label{sec:MGLC}
\label{sec:lowcoh}

\begin{theorem}
\label{thm:MGLC}
Let $F$ be a totally imaginary number field. Then
\[
\MGL_{*,*}(F;\Z) \cong H^{*,*}(F;\Z)\tensor L_*.
\]
Here $H^{p,q}(F;\Z)$ is a summand of $\MGL_{-p,-q}(F;\Z)$,
and $L_k$ is identified with $\MGL_{2k,k}(F;\Z)$.
\end{theorem}
\begin{proof}
The motivic cohomology $H^{p,q}(F;\Z)$ is zero for $p=0, q \neq 0$ or $p > 2$.
Hence, the slice spectral sequence collapses at the $E^1$-page, cf.~\Cref{fig:E1-page}.
There are no hidden extensions since in any fixed bidegree of $\MGL_{*,*}(F;\Z)$ the slice filtration has length $1$.
\end{proof}
\begin{theorem}
\label{thm:2coh}
Let $F$ be a field of characteristic not 2 with $2$-cohomological dimension less than or equal to 2,
such that $H^{p,q}(F;\Z/2)$ is finite for all $p,q$, and $H^{0,q}(F;\Z_2)$ is zero for $q \neq 0$.
Then
\[
\MGL_{*,*}(F;\Z_2) \cong H^{*,*}(F;\Z_2)\tensor L_*.
\]
Here $H^{p,q}(F;\Z_2)$ is a summand of $\MGL_{-p,-q}(F;\Z)$,
and $L_k$ is identified with $\MGL_{2k,k}(F;\Z)$.
\end{theorem}
In particular \Cref{thm:2coh} applies to local fields in characteristic $0$, cf.~\cite[Corollary 5.10]{Ormsby}, and finite fields of odd characteristic.

\section{The motivic Brown-Peterson spectrum and motivic Morava $K$-theory}
\label{sec:morava}
In this section we compute the motivic homotopy groups of the 2-completed (truncated) motivic Brown-Peterson spectrum $\BPGL$ over $\R$ at the prime 2, and of mod 2 Morava $K$-theory $K(n)$ over fields with low virtual cohomological dimension of characteristic not 2. The techniques are analogous to the previous sections. As part of the computation we give a description of the action of the mod 2 motivic Steenrod algebra on powers of $\tau$. Except when we discuss the general construction of $\BPGL$ and $K(n)$ we work at the prime $2$, so $\BPGL$ will be a $2$-local spectrum.
Most of the results over the real numbers in this section have been obtained previously by Yagita in \cite{Yagita:atiyah}.

For the construction of $\BPGL$ and $K(n)$, we follow  \cite{Hu-Kriz:remarks}, \cite{Vezzosi} and \cite[Chapter 4]{Ravenel}. See also \cite[Remark 5.3]{Spitzweck:slices}.
Localizing the bijection
$
[\MGL, \MGL]^\star \cong \MGL^{\star}[\![ b_1, b_2, \dots]\!]
$
\cite[Proposition 6.2]{NSO}
at $\ell$ we get the Quillen idempotent $e : \MGL_{(\ell)} \to \MGL_{(\ell)}$ \cite[Theorem 4.1.12]{Ravenel}.
The $\ell$-local motivic Brown-Peterson spectrum is defined as
\[
\BPGL = e\MGL_{{(\ell)}} = \hocolim(\MGL_{(\ell)} \xrightarrow{e} \MGL_{(\ell)} \xrightarrow{e} \dots).
\]
Since slices commute with homotopy colimits we get
\begin{align*}
\s_q(\BPGL)
& = \s_q(\hocolim(\MGL_{(\ell)} \xrightarrow{e} \MGL_{(\ell)} \xrightarrow{e} \dots)) \\
& = \hocolim(\s_q(\MGL_{(\ell)}) \xrightarrow{\s_q(e)} \s_q(\MGL_{(\ell)}) \xrightarrow{\s_q(e)} \dots) \\
& = \hocolim(\MZ_{(\ell)} \tensor L_q \xrightarrow{\s_q(e)} \MZ_{(\ell)} \tensor L_q \xrightarrow{\s_q(e)} \dots) \\
& = \hocolim(\MZ_{(\ell)} \tensor L_q \xrightarrow{1 \tensor e} \MZ_{(\ell)} \tensor L_q \xrightarrow{1\tensor e} \dots) \\
& = \MZ_{(\ell)} \tensor \colim(\Z_{(\ell)} \tensor L_q \xrightarrow{1 \tensor e} \Z_{(\ell)} \tensor L_q \xrightarrow{1 \tensor e} \dots) \\
& = \MZ_{(\ell)} \tensor (\Z_{(\ell)}[v_1, v_2, \dots])_q,
\end{align*}
where $\vert v_i \vert = \ell^i - 1$. %
This computation has been carried out in more detail in \cite{Levine-Tripathi}.

The truncated Brown-Peterson spectra are defined to be
\[
\BPGL\langle n \rangle = \BPGL/(v_{n+1}, v_{n+2}, \dots).
\]
The quotient is obtained by inductively forming the cofiber $\BPGL/(v_{n+1}, \dots, v_{n+k-1}, v_{n+k})$ and then taking a homotopy colimit, see \cite{Spitzweck:slices}.
In particular $\BPGL\langle 0 \rangle = \MZ_{(\ell)}$,
and $\BPGL\langle 0 \rangle = \kgl_{(\ell)} = \f_0(\KGL_{(\ell)})$ is the (very) effective cover of $\ell$-local algebraic $K$-theory \cite{Spitzweck:slices}.
The slices are
\[
\s_q(\BPGL\langle n \rangle) \cong \MZ_{(\ell)} \tensor (\Z_{(\ell)}[v_1, \dots, v_n])_q.
\]

The $n$th mod $\ell$ Morava $K$-theory is defined to be
\[
K(n) = \BPGL/(\ell, v_i \vert i \neq n)[v_n^{-1}].
\]
The slices are
\[
\s_q(K(n)) = \MZ/\ell \tensor (\Z/\ell[v_n^{\pm 1}])_q.
\]
Hence for $\s_q(K(n))$ to be nonzero $q$ must be a multiple of $\ell^{i} - 1$.
From now on we fix $\ell = 2$.

By adapting the proof of \Cref{thm:MGLR} we obtain the following calculation of the motivic homotopy groups of $\BPGL_{2}^{\wedge}$ and $\BPGL\langle n\rangle_{2}^{\wedge}$ over the reals.
This has been stated previously by Hu and Kriz in \cite{Hu-Kriz:real}, \cite{Hu-Kriz:remarks}.
\begin{theorem}
\label{thm:BPGL}
The motivic homotopy groups of $2$-complete $\BPGL$ and $\BPGL\langle n\rangle$ over the real numbers are
\[
\BPGL_\star(\R;\Z_2) \cong
\frac{\Z_2[\rho, z_{j,l}  \vert l, j \geq 0]}
{(2\rho, \rho^{2^{j+1} - 1}z_{j,l}, z_{k,l}z_{a,b} = z_{k,l+b2^{a-k}}z_{a,0}, \text{ for } a \geq k)},
\]
\[
\BPGL\langle n \rangle_\star(\R;\Z_2) \cong
\frac{\Z_2[\rho, z_{j,l}, u^{2^n}  \vert l,j \geq 0, n \geq j ]}
{(2\rho, \rho^{2^{j+1} - 1}z_{j,l}, z_{k,l}z_{a,b} = z_{k,l+b2^{a-k}}z_{a,0}, \text{ for } a \geq k, z_{j,l}u^{2^n} = z_{j,l+2^{n-j}})}.
\]
\end{theorem}
\begin{proof}
Since $\BPGL$ is a summand of $\MGL_{(2)}$, the differentials take the same form.
The differentials of $\BPGL\langle n \rangle$ are determined by the differentials of $\BPGL$.
\end{proof}

Next we compute the $n$th mod 2 Morava $K$-theory of fields not of characteristic 2 with virtual cohomological dimension $\vcd(F) < 2(2^n-1)$.

\begin{lemma}[\protect{\cite[Lemma 8.11]{Hoyois:fromto}}]
\label{lem:veff-agree}
The effective and the very effective slice filtrations \cite{Bachmann:veff} are the same on
$\MGL$, $\BPGL$, $\BPGL\langle n \rangle$ and $K(n)$.
That is,
$
\f_q(\E) = \vf_q(\E),
$
for $\E$ any of the above spectra.
\end{lemma}
\begin{proof}
By \cite[Lemma 8.11]{Hoyois:fromto}, $\f_q(\E)$ is $q$-connective.
Hence, $\f_q(\E)$ actually lies in $\SH^{\veff}(q)$,
but then it is equal to $\vf_q(\E)$ by uniqueness of adjoints.
\end{proof}

\begin{lemma}
\label{lem:veff-real}
Let $\E$ be a motivic spectrum in $\SH(\C)$ whose complex realization
$R_\C(\E)$ has only cells in even dimensions.
Then the complex realization of the very effective slice filtration of $\E$ is twice the Postnikov filtration of the complex realization of $\E$.
That is,
$
R_\C(\vf_q(\E)) = (R_\C(\E))_{\geq 2q}.
$
Hence complex realization induces a map from the very effective slice spectral sequence of $\E$ to the Atiyah-Hirzebruch spectral sequence of $R_\C(\E)$ \cite[Theorem 3.3]{Maunder}.
\end{lemma}
\begin{proof}
By construction \cite[Section 4]{Bachmann:veff}
\[
\SH(\C)^\veff(q) = T^{\wedge q} \wedge \SH(\C)^\eff_{\geq 0} = 
T^{\wedge q} \wedge \{\E \in \SH(\C)^\eff \vert \underline{\pi}_{i, 0}(\E) = 0, i < 0\},
\]
where $\underline{\pi}_{i,0}(\E)$ is the Nisnevich sheaf associated to the presheaf $X \mapsto [\Sigma^\infty X_+ \wedge S^{i,0}, \E]$.
Since $R_\C(T) \simeq S^2$, we have
$R_\C(\SH(k)^\veff(q)) = \SH_{\geq 2q}$, i.e., the usual Postnikov filtration of $\SH$ at twice the usual speed.
If $R_\C(\E)_{\geq 2q} = R_\C(\E)_{\geq 2q + 1}$,
then $R_\C(\vf_q(\E)) = (R_\C(\E))_{\geq 2q}$ by definition.
\end{proof}
Note that the condition of \Cref{lem:veff-real} is satisfied by the spectra
$\MGL, \BPGL, \BPGL\langle n \rangle$ and $K(n)$.
The homotopy groups of their complex realizations are only nonzero in even degrees.
Hence for each of these spectra we have a well defined map from their slice spectral sequence to the Atiyah-Hirzebruch spectral sequence of their complex realizations.

\begin{lemma}
\label{lem:Kn-diffs}
Over any field $F$ of characteristic not 2 the $d^i$-differential in the slice spectral sequence for $K(n)$ is trivial for $i < r = 2^n-1$ and
$d^{r} = Q_n + \Psi$. Here $Q_n$ is the $n$th motivic Milnor primitive, the dual of $\tau_n$ in the mod 2 dual motivic Steenrod algebra, while $\Psi$ is some element of the motivic Steenrod algebra which acts as zero on $H^\star(F;\Z/2)$.
\end{lemma}
\begin{proof}
Since the distance between the nonzero slices are $2^n-1$ the first possibly nonzero differential is $d^{2^n-1}$
which is an element of $\AAA^{2^{n+1}-1,2^n-1}$.
Elements in this bidegree are of the form
\[
aQ_n + bQ_0\hat{f} + \phi \Phi,
\]
where $a, b \in \Z/2$, $\phi \in H^{1,1}(\R;\Z/2) = \Z/2\{\rho\}$, $\Phi \in \AAA^\star$, $f$ is a polynomial in the dual motivic Steenrod algebra and $\hat{f}$ its dual.

By complex realization $a = 1$ and $b = 0$ by \Cref{lem:veff-real}, \Cref{lem:veff-agree} and \cite[Lemma 2.1]{Yagita:Morava} (note that  \cite[Lemma 2.1]{Yagita:Morava} is only stated for odd primes, but because of the (noncommutative) ring structure on $K(n)$, cf.~\cite[p.~1756]{Nassau}, Yagita's proof works for $p=2$). %
It remains to show that $\Psi = \phi \Phi$ acts as zero on $H^\star(F;\Z/2)$.
Using \Cref{thm:BPGL}, the map $\BPGL \to K(n)$ implies that $d^{2^n-1}(\tau^{2^n}) = \rho^{2^{n+1}-1} = Q_n(\tau^{2^n})$, where the last equality is by \Cref{cor:Qn-action}. Note that \Cref{thm:BPGL} is only stated for $\R$, but by base change from $\Q$ the differential takes the same form over any field of characteristic 0, since $H^{p,q}(\R;\Z/2) \cong H^{p,q}(\Q;\Z/2)$ when $p > 2$. For fields of positive characteristic there is nothing to show, since then $\rho^3 = 0$.
More generally we get that $d^{2^n-1}(\tau^{u2^n}) = Q_n(\tau^{u2^n})$, so $\Psi(\tau^{u2^n}) = 0$.
Since $\Psi$ has bidegree $(2^{n+1}-1, 2^n-1)$ we have $\Psi(\tau^k) = 0$, since $H^{2^{n+1}-1, 2^n-1 + k} = 0$ when $k < 2^n$.
Next we use the $\BPGL$-module structure $\BPGL \wedge K(n) \to K(n)$ and the Leibniz rule to conclude that
$\Psi(\tau^{k + u2^n}) = 0$ for $0 \leq k < 2^n$ and all $u$.
Indeed,
\begin{align*}
Q_n(\tau^{k + u2^n}) + \Psi(\tau^{k + u 2^n})
&=  d^{2^n-1}(\tau^{u2^n}\wedge \tau^k)
= d^{2^n-1}(\tau^{u2^n})\wedge \tau^k + \tau^{u2^n} \wedge d^{2^n-1}(\tau^{k}) \\
&= Q_n(\tau^{u2^n})\wedge \tau^k + 0.
\end{align*}
\end{proof}

\begin{remark}
The condition $(d^{2^n-1})^2 = 0$ is not sufficient to determine if $\phi = 0$.
Indeed, for $n=2$ we could have $d^3 = Q_2 + \rho Q_0 Q_1\Sq^2$.
One approach to decide if $d^{2^n-1} = Q_n$ would be to compute the motivic Morava $K$-theory of products of the motivic classifying space $B\mu_2$.
\end{remark}

Over the real numbers the next result have been obtained in \cite[p.~10]{Hu-Kriz:remarks} and \cite[Theorem 6.4]{Yagita:atiyah}.
\begin{theorem}
\label{lem:Kn-coeff}
\label{thm:Kn}
The associated graded of the motivic homotopy groups of $K(n)$ over a field $F$ of characteristic not 2 with virtual cohomological dimension $\vcd(F) < 2(2^n-1)$ is
\begin{align*}
E^\infty(K_\star(n)) \cong &
\bigoplus_{i\geq0, {i \choose 2^n} \equiv 0}k_*(F)[v_n^{\pm 1}]/\rho^{2^{n+1}-1}\{\tau^i\}\\
&\oplus
\bigoplus_{i\geq0, {i \choose 2^n} \equiv 1}\ker(\rho^{2^{n+1}-1} : k_*(F)[v_n^{\pm 1}] \to k_*(F)[v_n^{\pm 1}])\{\tau^i \}.
\end{align*}
Here $\rho_{} \in E^\infty_{-1,0,-1}(K(n)), \tau \in E^\infty_{0,0,-1}(K(n)), v_n \in E^\infty_{2(2^{n}-1), 2^n-1, 2^{n}-1}(K(n))$.
When $p - 2w < 0$, $K_{p,w}(n) = 0$.
See \Cref{fig:K(n)} for a picture of $K_{\star}(n)(\R)$.
\end{theorem}
\begin{proof}
As a module over mod $2$ Milnor $K$-theory $k_*(F)$ we have
\[
E^1(K(n)) = \pi_{*,*}\s_*(K(n)) = \oplus_{i\geq0}k_*(F)[v_n^{\pm 1}]\{\tau^i\}.
\]
The differentials are $k_*(F)$- and $v_n$-linear,
the module structure being induced by the $\S$-module and $\BPGL$-module structure on $K(n)$.
By \Cref{lem:Kn-diffs} we may assume $d^{2^n-1} = Q_n$, so \Cref{cor:Qn-action} implies that the $E^{2^n}$-page is
\begin{align*}
E^{2^n}(K(n)) =& 
\bigoplus_{i\geq0, {i \choose 2^n} \equiv 0}k_*(F)[v_n^{\pm 1}]/\rho^{2^{n+1}-1}\{\tau^i\}\\
&\oplus
\bigoplus_{i\geq0, {i \choose 2^n} \equiv 1}\ker\left(\rho^{2^{n+1}-1} : k_*(F)[v_n^{\pm 1}] \to k_*(F)[v_n^{\pm 1}]\right)\{\tau^i \}.
\end{align*}
Over a field $F$ with $\vcd(F) < 2(2^n-1)$ the spectral sequence collapses at the $E^{2^n}$-page for degree reasons,
i.e., $E^{2^n}(K(n)) = E^\infty(K(n))$.
Indeed,
the $d^r$-differential is a map
$d^r : E^{r}_{p,q,w}  \to E^{r}_{p-1,q+r,w}$,
where $E^r_{p,q,w}$ is a subquotient of $H^{2q-p,q-w}(F;\Z/2) \tensor (\Z[v_n^{\pm}])_q$.
At an earlier stage the group $E^{r'}_{p-1,q+r,w}$ could be the target of a nonzero $d^{r'}$-differential from $E^{r'}_{p,q+r-r',w}$.
If $2(q+r-r') - p = 2q - p + 2(r-r') \geq 0 + 2(r - r') \geq 2(2^n-1) > \vcd(F)$ then the nonzero $d^{r'}$-differential is an isomorphism, hence the target of the $d^r$-differential is zero (note that $r$ is necessarily a multiple of $2^n-1$ and $r'=2^n-1$). Hence there are no nonzero $d^r$-differentials for $r > 2^n-1$.
\end{proof}
\begin{remark}
For $K(1) = \KGL/2$ a finer analysis can be carried out over number fields, cf.~\cite{Rognes-Weibel}, \cite[Section 4]{KRO}.
\end{remark}
\begin{remark}
\label{rmk:Kn-exts}
Over the real numbers we are unable to determine the additive extension
\[
E^\infty_{0,*,-(2^n-1)} = \Z/2\{\tau^{2^n-1}, \rho^{2^{n+1}-2}v_n \}.
\]
Based on the $n=0,1$ cases, we may guess that multiplication by $2$ on $K(n)$ induces the map $v_n\rho^nQ_0Q_1\dots Q_{n-1}$ on slices, which would detect that the extension is nonsplit.
\end{remark}

\begin{figure}
  \centering

\begin{tikzpicture}[font=\tiny,scale=0.4]
        \draw[help lines,xstep=1.0,ystep=1.0] (-0.5,0.0) grid (20.5,20.5);

{\draw[fill] (0.0,8.0) circle (2pt);}
{\draw[fill] (0.0,7.0) circle (2pt);}
{\draw (0.0,8.0) -- (0.0,7.0);}
{\draw[fill] (0.0,6.0) circle (2pt);}
{\draw (0.0,7.0) -- (0.0,6.0);}
{\draw[fill] (0.0,5.0) circle (2pt);}
{\draw (0.0,6.0) -- (0.0,5.0);}
{\draw[fill] (0.0,0.0) circle (2pt);}
{\draw[fill] (6.0,11.0) circle (2pt);}
{\draw[fill] (6.0,10.0) circle (2pt);}
{\draw (6.0,11.0) -- (6.0,10.0);}
{\draw[fill] (6.0,9.0) circle (2pt);}
{\draw (6.0,10.0) -- (6.0,9.0);}
{\draw[fill] (6.0,8.0) circle (2pt);}
{\draw (6.0,9.0) -- (6.0,8.0);}
{\draw[fill] (5.0,10.0) circle (2pt);}
{\draw (6.0,11.0) -- (5.0,10.0);}
{\draw[fill] (5.0,9.0) circle (2pt);}
{\draw (5.0,10.0) -- (5.0,9.0);}
{\draw (6.0,10.0) -- (5.0,9.0);}
{\draw[fill] (5.0,8.0) circle (2pt);}
{\draw (5.0,9.0) -- (5.0,8.0);}
{\draw (6.0,9.0) -- (5.0,8.0);}
{\draw[fill] (5.0,7.0) circle (2pt);}
{\draw (5.0,8.0) -- (5.0,7.0);}
{\draw (6.0,8.0) -- (5.0,7.0);}
{\draw[fill] (4.0,9.0) circle (2pt);}
{\draw (5.0,10.0) -- (4.0,9.0);}
{\draw[fill] (4.0,8.0) circle (2pt);}
{\draw (4.0,9.0) -- (4.0,8.0);}
{\draw (5.0,9.0) -- (4.0,8.0);}
{\draw[fill] (4.0,7.0) circle (2pt);}
{\draw (4.0,8.0) -- (4.0,7.0);}
{\draw (5.0,8.0) -- (4.0,7.0);}
{\draw[fill] (4.0,6.0) circle (2pt);}
{\draw (4.0,7.0) -- (4.0,6.0);}
{\draw (5.0,7.0) -- (4.0,6.0);}
{\draw[fill] (3.0,8.0) circle (2pt);}
{\draw (4.0,9.0) -- (3.0,8.0);}
{\draw[fill] (3.0,7.0) circle (2pt);}
{\draw (3.0,8.0) -- (3.0,7.0);}
{\draw (4.0,8.0) -- (3.0,7.0);}
{\draw[fill] (3.0,6.0) circle (2pt);}
{\draw (3.0,7.0) -- (3.0,6.0);}
{\draw (4.0,7.0) -- (3.0,6.0);}
{\draw[fill] (3.0,5.0) circle (2pt);}
{\draw (3.0,6.0) -- (3.0,5.0);}
{\draw (4.0,6.0) -- (3.0,5.0);}
{\draw[fill] (2.0,7.0) circle (2pt);}
{\draw (3.0,8.0) -- (2.0,7.0);}
{\draw[fill] (2.0,6.0) circle (2pt);}
{\draw (2.0,7.0) -- (2.0,6.0);}
{\draw (3.0,7.0) -- (2.0,6.0);}
{\draw[fill] (2.0,5.0) circle (2pt);}
{\draw (2.0,6.0) -- (2.0,5.0);}
{\draw (3.0,6.0) -- (2.0,5.0);}
{\draw[fill] (2.0,4.0) circle (2pt);}
{\draw (2.0,5.0) -- (2.0,4.0);}
{\draw (3.0,5.0) -- (2.0,4.0);}
{\draw[fill] (1.0,6.0) circle (2pt);}
{\draw (2.0,7.0) -- (1.0,6.0);}
{\draw[fill] (1.0,5.0) circle (2pt);}
{\draw (1.0,6.0) -- (1.0,5.0);}
{\draw (2.0,6.0) -- (1.0,5.0);}
{\draw[fill] (1.0,4.0) circle (2pt);}
{\draw (1.0,5.0) -- (1.0,4.0);}
{\draw (2.0,5.0) -- (1.0,4.0);}
{\draw[fill] (1.0,3.0) circle (2pt);}
{\draw (1.0,4.0) -- (1.0,3.0);}
{\draw (2.0,4.0) -- (1.0,3.0);}
{\draw[fill] (0.0,5.0) circle (2pt);}
{\draw (1.0,6.0) -- (0.0,5.0);}
{\draw[fill] (0.0,4.0) circle (2pt);}
{\draw (0.0,5.0) -- (0.0,4.0);}
{\draw (1.0,5.0) -- (0.0,4.0);}
{\draw[fill] (0.0,3.0) circle (2pt);}
{\draw (0.0,4.0) -- (0.0,3.0);}
{\draw (1.0,4.0) -- (0.0,3.0);}
{\draw[fill] (0.0,2.0) circle (2pt);}
{\draw (0.0,3.0) -- (0.0,2.0);}
{\draw (1.0,3.0) -- (0.0,2.0);}
{\draw[fill] (6.0,3.0) circle (2pt);}
{\draw[fill] (6.0,2.0) circle (2pt);}
{\draw (6.0,3.0) -- (6.0,2.0);}
{\draw[fill] (6.0,1.0) circle (2pt);}
{\draw (6.0,2.0) -- (6.0,1.0);}
{\draw[fill] (6.0,0.0) circle (2pt);}
{\draw (6.0,1.0) -- (6.0,0.0);}
{\draw[fill] (5.0,2.0) circle (2pt);}
{\draw (6.0,3.0) -- (5.0,2.0);}
{\draw[fill] (5.0,1.0) circle (2pt);}
{\draw (5.0,2.0) -- (5.0,1.0);}
{\draw (6.0,2.0) -- (5.0,1.0);}
{\draw[fill] (5.0,0.0) circle (2pt);}
{\draw (5.0,1.0) -- (5.0,0.0);}
{\draw (6.0,1.0) -- (5.0,0.0);}
{\draw[fill] (4.0,1.0) circle (2pt);}
{\draw (5.0,2.0) -- (4.0,1.0);}
{\draw[fill] (4.0,0.0) circle (2pt);}
{\draw (4.0,1.0) -- (4.0,0.0);}
{\draw (5.0,1.0) -- (4.0,0.0);}
{\draw[fill] (3.0,0.0) circle (2pt);}
{\draw (4.0,1.0) -- (3.0,0.0);}
{\draw[fill] (12.0,14.0) circle (2pt);}
{\draw[fill] (12.0,13.0) circle (2pt);}
{\draw (12.0,14.0) -- (12.0,13.0);}
{\draw[fill] (12.0,12.0) circle (2pt);}
{\draw (12.0,13.0) -- (12.0,12.0);}
{\draw[fill] (12.0,11.0) circle (2pt);}
{\draw (12.0,12.0) -- (12.0,11.0);}
{\draw[fill] (11.0,13.0) circle (2pt);}
{\draw (12.0,14.0) -- (11.0,13.0);}
{\draw[fill] (11.0,12.0) circle (2pt);}
{\draw (11.0,13.0) -- (11.0,12.0);}
{\draw (12.0,13.0) -- (11.0,12.0);}
{\draw[fill] (11.0,11.0) circle (2pt);}
{\draw (11.0,12.0) -- (11.0,11.0);}
{\draw (12.0,12.0) -- (11.0,11.0);}
{\draw[fill] (11.0,10.0) circle (2pt);}
{\draw (11.0,11.0) -- (11.0,10.0);}
{\draw (12.0,11.0) -- (11.0,10.0);}
{\draw[fill] (10.0,12.0) circle (2pt);}
{\draw (11.0,13.0) -- (10.0,12.0);}
{\draw[fill] (10.0,11.0) circle (2pt);}
{\draw (10.0,12.0) -- (10.0,11.0);}
{\draw (11.0,12.0) -- (10.0,11.0);}
{\draw[fill] (10.0,10.0) circle (2pt);}
{\draw (10.0,11.0) -- (10.0,10.0);}
{\draw (11.0,11.0) -- (10.0,10.0);}
{\draw[fill] (10.0,9.0) circle (2pt);}
{\draw (10.0,10.0) -- (10.0,9.0);}
{\draw (11.0,10.0) -- (10.0,9.0);}
{\draw[fill] (9.0,11.0) circle (2pt);}
{\draw (10.0,12.0) -- (9.0,11.0);}
{\draw[fill] (9.0,10.0) circle (2pt);}
{\draw (9.0,11.0) -- (9.0,10.0);}
{\draw (10.0,11.0) -- (9.0,10.0);}
{\draw[fill] (9.0,9.0) circle (2pt);}
{\draw (9.0,10.0) -- (9.0,9.0);}
{\draw (10.0,10.0) -- (9.0,9.0);}
{\draw[fill] (9.0,8.0) circle (2pt);}
{\draw (9.0,9.0) -- (9.0,8.0);}
{\draw (10.0,9.0) -- (9.0,8.0);}
{\draw[fill] (8.0,10.0) circle (2pt);}
{\draw (9.0,11.0) -- (8.0,10.0);}
{\draw[fill] (8.0,9.0) circle (2pt);}
{\draw (8.0,10.0) -- (8.0,9.0);}
{\draw (9.0,10.0) -- (8.0,9.0);}
{\draw[fill] (8.0,8.0) circle (2pt);}
{\draw (8.0,9.0) -- (8.0,8.0);}
{\draw (9.0,9.0) -- (8.0,8.0);}
{\draw[fill] (8.0,7.0) circle (2pt);}
{\draw (8.0,8.0) -- (8.0,7.0);}
{\draw (9.0,8.0) -- (8.0,7.0);}
{\draw[fill] (7.0,9.0) circle (2pt);}
{\draw (8.0,10.0) -- (7.0,9.0);}
{\draw[fill] (7.0,8.0) circle (2pt);}
{\draw (7.0,9.0) -- (7.0,8.0);}
{\draw (8.0,9.0) -- (7.0,8.0);}
{\draw[fill] (7.0,7.0) circle (2pt);}
{\draw (7.0,8.0) -- (7.0,7.0);}
{\draw (8.0,8.0) -- (7.0,7.0);}
{\draw[fill] (7.0,6.0) circle (2pt);}
{\draw (7.0,7.0) -- (7.0,6.0);}
{\draw (8.0,7.0) -- (7.0,6.0);}
{\draw[fill] (6.0,8.0) circle (2pt);}
{\draw (7.0,9.0) -- (6.0,8.0);}
{\draw[fill] (6.0,7.0) circle (2pt);}
{\draw (6.0,8.0) -- (6.0,7.0);}
{\draw (7.0,8.0) -- (6.0,7.0);}
{\draw[fill] (6.0,6.0) circle (2pt);}
{\draw (6.0,7.0) -- (6.0,6.0);}
{\draw (7.0,7.0) -- (6.0,6.0);}
{\draw[fill] (6.0,5.0) circle (2pt);}
{\draw (6.0,6.0) -- (6.0,5.0);}
{\draw (7.0,6.0) -- (6.0,5.0);}
{\draw[fill] (12.0,6.0) circle (2pt);}
{\draw[fill] (12.0,5.0) circle (2pt);}
{\draw (12.0,6.0) -- (12.0,5.0);}
{\draw[fill] (12.0,4.0) circle (2pt);}
{\draw (12.0,5.0) -- (12.0,4.0);}
{\draw[fill] (12.0,3.0) circle (2pt);}
{\draw (12.0,4.0) -- (12.0,3.0);}
{\draw[fill] (11.0,5.0) circle (2pt);}
{\draw (12.0,6.0) -- (11.0,5.0);}
{\draw[fill] (11.0,4.0) circle (2pt);}
{\draw (11.0,5.0) -- (11.0,4.0);}
{\draw (12.0,5.0) -- (11.0,4.0);}
{\draw[fill] (11.0,3.0) circle (2pt);}
{\draw (11.0,4.0) -- (11.0,3.0);}
{\draw (12.0,4.0) -- (11.0,3.0);}
{\draw[fill] (11.0,2.0) circle (2pt);}
{\draw (11.0,3.0) -- (11.0,2.0);}
{\draw (12.0,3.0) -- (11.0,2.0);}
{\draw[fill] (10.0,4.0) circle (2pt);}
{\draw (11.0,5.0) -- (10.0,4.0);}
{\draw[fill] (10.0,3.0) circle (2pt);}
{\draw (10.0,4.0) -- (10.0,3.0);}
{\draw (11.0,4.0) -- (10.0,3.0);}
{\draw[fill] (10.0,2.0) circle (2pt);}
{\draw (10.0,3.0) -- (10.0,2.0);}
{\draw (11.0,3.0) -- (10.0,2.0);}
{\draw[fill] (10.0,1.0) circle (2pt);}
{\draw (10.0,2.0) -- (10.0,1.0);}
{\draw (11.0,2.0) -- (10.0,1.0);}
{\draw[fill] (9.0,3.0) circle (2pt);}
{\draw (10.0,4.0) -- (9.0,3.0);}
{\draw[fill] (9.0,2.0) circle (2pt);}
{\draw (9.0,3.0) -- (9.0,2.0);}
{\draw (10.0,3.0) -- (9.0,2.0);}
{\draw[fill] (9.0,1.0) circle (2pt);}
{\draw (9.0,2.0) -- (9.0,1.0);}
{\draw (10.0,2.0) -- (9.0,1.0);}
{\draw[fill] (9.0,0.0) circle (2pt);}
{\draw (9.0,1.0) -- (9.0,0.0);}
{\draw (10.0,1.0) -- (9.0,0.0);}
{\draw[fill] (8.0,2.0) circle (2pt);}
{\draw (9.0,3.0) -- (8.0,2.0);}
{\draw[fill] (8.0,1.0) circle (2pt);}
{\draw (8.0,2.0) -- (8.0,1.0);}
{\draw (9.0,2.0) -- (8.0,1.0);}
{\draw[fill] (8.0,0.0) circle (2pt);}
{\draw (8.0,1.0) -- (8.0,0.0);}
{\draw (9.0,1.0) -- (8.0,0.0);}
{\draw[fill] (7.0,1.0) circle (2pt);}
{\draw (8.0,2.0) -- (7.0,1.0);}
{\draw[fill] (7.0,0.0) circle (2pt);}
{\draw (7.0,1.0) -- (7.0,0.0);}
{\draw (8.0,1.0) -- (7.0,0.0);}
{\draw[fill] (6.0,0.0) circle (2pt);}
{\draw (7.0,1.0) -- (6.0,0.0);}
{\draw[fill] (18.0,17.0) circle (2pt);}
{\draw[fill] (18.0,16.0) circle (2pt);}
{\draw (18.0,17.0) -- (18.0,16.0);}
{\draw[fill] (18.0,15.0) circle (2pt);}
{\draw (18.0,16.0) -- (18.0,15.0);}
{\draw[fill] (18.0,14.0) circle (2pt);}
{\draw (18.0,15.0) -- (18.0,14.0);}
{\draw[fill] (17.0,16.0) circle (2pt);}
{\draw (18.0,17.0) -- (17.0,16.0);}
{\draw[fill] (17.0,15.0) circle (2pt);}
{\draw (17.0,16.0) -- (17.0,15.0);}
{\draw (18.0,16.0) -- (17.0,15.0);}
{\draw[fill] (17.0,14.0) circle (2pt);}
{\draw (17.0,15.0) -- (17.0,14.0);}
{\draw (18.0,15.0) -- (17.0,14.0);}
{\draw[fill] (17.0,13.0) circle (2pt);}
{\draw (17.0,14.0) -- (17.0,13.0);}
{\draw (18.0,14.0) -- (17.0,13.0);}
{\draw[fill] (16.0,15.0) circle (2pt);}
{\draw (17.0,16.0) -- (16.0,15.0);}
{\draw[fill] (16.0,14.0) circle (2pt);}
{\draw (16.0,15.0) -- (16.0,14.0);}
{\draw (17.0,15.0) -- (16.0,14.0);}
{\draw[fill] (16.0,13.0) circle (2pt);}
{\draw (16.0,14.0) -- (16.0,13.0);}
{\draw (17.0,14.0) -- (16.0,13.0);}
{\draw[fill] (16.0,12.0) circle (2pt);}
{\draw (16.0,13.0) -- (16.0,12.0);}
{\draw (17.0,13.0) -- (16.0,12.0);}
{\draw[fill] (15.0,14.0) circle (2pt);}
{\draw (16.0,15.0) -- (15.0,14.0);}
{\draw[fill] (15.0,13.0) circle (2pt);}
{\draw (15.0,14.0) -- (15.0,13.0);}
{\draw (16.0,14.0) -- (15.0,13.0);}
{\draw[fill] (15.0,12.0) circle (2pt);}
{\draw (15.0,13.0) -- (15.0,12.0);}
{\draw (16.0,13.0) -- (15.0,12.0);}
{\draw[fill] (15.0,11.0) circle (2pt);}
{\draw (15.0,12.0) -- (15.0,11.0);}
{\draw (16.0,12.0) -- (15.0,11.0);}
{\draw[fill] (14.0,13.0) circle (2pt);}
{\draw (15.0,14.0) -- (14.0,13.0);}
{\draw[fill] (14.0,12.0) circle (2pt);}
{\draw (14.0,13.0) -- (14.0,12.0);}
{\draw (15.0,13.0) -- (14.0,12.0);}
{\draw[fill] (14.0,11.0) circle (2pt);}
{\draw (14.0,12.0) -- (14.0,11.0);}
{\draw (15.0,12.0) -- (14.0,11.0);}
{\draw[fill] (14.0,10.0) circle (2pt);}
{\draw (14.0,11.0) -- (14.0,10.0);}
{\draw (15.0,11.0) -- (14.0,10.0);}
{\draw[fill] (13.0,12.0) circle (2pt);}
{\draw (14.0,13.0) -- (13.0,12.0);}
{\draw[fill] (13.0,11.0) circle (2pt);}
{\draw (13.0,12.0) -- (13.0,11.0);}
{\draw (14.0,12.0) -- (13.0,11.0);}
{\draw[fill] (13.0,10.0) circle (2pt);}
{\draw (13.0,11.0) -- (13.0,10.0);}
{\draw (14.0,11.0) -- (13.0,10.0);}
{\draw[fill] (13.0,9.0) circle (2pt);}
{\draw (13.0,10.0) -- (13.0,9.0);}
{\draw (14.0,10.0) -- (13.0,9.0);}
{\draw[fill] (12.0,11.0) circle (2pt);}
{\draw (13.0,12.0) -- (12.0,11.0);}
{\draw[fill] (12.0,10.0) circle (2pt);}
{\draw (12.0,11.0) -- (12.0,10.0);}
{\draw (13.0,11.0) -- (12.0,10.0);}
{\draw[fill] (12.0,9.0) circle (2pt);}
{\draw (12.0,10.0) -- (12.0,9.0);}
{\draw (13.0,10.0) -- (12.0,9.0);}
{\draw[fill] (12.0,8.0) circle (2pt);}
{\draw (12.0,9.0) -- (12.0,8.0);}
{\draw (13.0,9.0) -- (12.0,8.0);}
{\draw[fill] (18.0,9.0) circle (2pt);}
{\draw[fill] (18.0,8.0) circle (2pt);}
{\draw (18.0,9.0) -- (18.0,8.0);}
{\draw[fill] (18.0,7.0) circle (2pt);}
{\draw (18.0,8.0) -- (18.0,7.0);}
{\draw[fill] (18.0,6.0) circle (2pt);}
{\draw (18.0,7.0) -- (18.0,6.0);}
{\draw[fill] (17.0,8.0) circle (2pt);}
{\draw (18.0,9.0) -- (17.0,8.0);}
{\draw[fill] (17.0,7.0) circle (2pt);}
{\draw (17.0,8.0) -- (17.0,7.0);}
{\draw (18.0,8.0) -- (17.0,7.0);}
{\draw[fill] (17.0,6.0) circle (2pt);}
{\draw (17.0,7.0) -- (17.0,6.0);}
{\draw (18.0,7.0) -- (17.0,6.0);}
{\draw[fill] (17.0,5.0) circle (2pt);}
{\draw (17.0,6.0) -- (17.0,5.0);}
{\draw (18.0,6.0) -- (17.0,5.0);}
{\draw[fill] (16.0,7.0) circle (2pt);}
{\draw (17.0,8.0) -- (16.0,7.0);}
{\draw[fill] (16.0,6.0) circle (2pt);}
{\draw (16.0,7.0) -- (16.0,6.0);}
{\draw (17.0,7.0) -- (16.0,6.0);}
{\draw[fill] (16.0,5.0) circle (2pt);}
{\draw (16.0,6.0) -- (16.0,5.0);}
{\draw (17.0,6.0) -- (16.0,5.0);}
{\draw[fill] (16.0,4.0) circle (2pt);}
{\draw (16.0,5.0) -- (16.0,4.0);}
{\draw (17.0,5.0) -- (16.0,4.0);}
{\draw[fill] (15.0,6.0) circle (2pt);}
{\draw (16.0,7.0) -- (15.0,6.0);}
{\draw[fill] (15.0,5.0) circle (2pt);}
{\draw (15.0,6.0) -- (15.0,5.0);}
{\draw (16.0,6.0) -- (15.0,5.0);}
{\draw[fill] (15.0,4.0) circle (2pt);}
{\draw (15.0,5.0) -- (15.0,4.0);}
{\draw (16.0,5.0) -- (15.0,4.0);}
{\draw[fill] (15.0,3.0) circle (2pt);}
{\draw (15.0,4.0) -- (15.0,3.0);}
{\draw (16.0,4.0) -- (15.0,3.0);}
{\draw[fill] (14.0,5.0) circle (2pt);}
{\draw (15.0,6.0) -- (14.0,5.0);}
{\draw[fill] (14.0,4.0) circle (2pt);}
{\draw (14.0,5.0) -- (14.0,4.0);}
{\draw (15.0,5.0) -- (14.0,4.0);}
{\draw[fill] (14.0,3.0) circle (2pt);}
{\draw (14.0,4.0) -- (14.0,3.0);}
{\draw (15.0,4.0) -- (14.0,3.0);}
{\draw[fill] (14.0,2.0) circle (2pt);}
{\draw (14.0,3.0) -- (14.0,2.0);}
{\draw (15.0,3.0) -- (14.0,2.0);}
{\draw[fill] (13.0,4.0) circle (2pt);}
{\draw (14.0,5.0) -- (13.0,4.0);}
{\draw[fill] (13.0,3.0) circle (2pt);}
{\draw (13.0,4.0) -- (13.0,3.0);}
{\draw (14.0,4.0) -- (13.0,3.0);}
{\draw[fill] (13.0,2.0) circle (2pt);}
{\draw (13.0,3.0) -- (13.0,2.0);}
{\draw (14.0,3.0) -- (13.0,2.0);}
{\draw[fill] (13.0,1.0) circle (2pt);}
{\draw (13.0,2.0) -- (13.0,1.0);}
{\draw (14.0,2.0) -- (13.0,1.0);}
{\draw[fill] (12.0,3.0) circle (2pt);}
{\draw (13.0,4.0) -- (12.0,3.0);}
{\draw[fill] (12.0,2.0) circle (2pt);}
{\draw (12.0,3.0) -- (12.0,2.0);}
{\draw (13.0,3.0) -- (12.0,2.0);}
{\draw[fill] (12.0,1.0) circle (2pt);}
{\draw (12.0,2.0) -- (12.0,1.0);}
{\draw (13.0,2.0) -- (12.0,1.0);}
{\draw[fill] (12.0,0.0) circle (2pt);}
{\draw (12.0,1.0) -- (12.0,0.0);}
{\draw (13.0,1.0) -- (12.0,0.0);}
{\draw[fill] (18.0,1.0) circle (2pt);}
{\draw[fill] (18.0,0.0) circle (2pt);}
{\draw (18.0,1.0) -- (18.0,0.0);}
{\draw[fill] (17.0,0.0) circle (2pt);}
{\draw (18.0,1.0) -- (17.0,0.0);}
{\draw[fill] (20.0,16.0) circle (2pt);}
{\draw[fill] (20.0,15.0) circle (2pt);}
{\draw (20.0,16.0) -- (20.0,15.0);}
{\draw[fill] (20.0,14.0) circle (2pt);}
{\draw (20.0,15.0) -- (20.0,14.0);}
{\draw[fill] (20.0,13.0) circle (2pt);}
{\draw (20.0,14.0) -- (20.0,13.0);}
{\draw[fill] (19.0,15.0) circle (2pt);}
{\draw (20.0,16.0) -- (19.0,15.0);}
{\draw[fill] (19.0,14.0) circle (2pt);}
{\draw (19.0,15.0) -- (19.0,14.0);}
{\draw (20.0,15.0) -- (19.0,14.0);}
{\draw[fill] (19.0,13.0) circle (2pt);}
{\draw (19.0,14.0) -- (19.0,13.0);}
{\draw (20.0,14.0) -- (19.0,13.0);}
{\draw[fill] (19.0,12.0) circle (2pt);}
{\draw (19.0,13.0) -- (19.0,12.0);}
{\draw (20.0,13.0) -- (19.0,12.0);}
{\draw[fill] (18.0,14.0) circle (2pt);}
{\draw (19.0,15.0) -- (18.0,14.0);}
{\draw[fill] (18.0,13.0) circle (2pt);}
{\draw (18.0,14.0) -- (18.0,13.0);}
{\draw (19.0,14.0) -- (18.0,13.0);}
{\draw[fill] (18.0,12.0) circle (2pt);}
{\draw (18.0,13.0) -- (18.0,12.0);}
{\draw (19.0,13.0) -- (18.0,12.0);}
{\draw[fill] (18.0,11.0) circle (2pt);}
{\draw (18.0,12.0) -- (18.0,11.0);}
{\draw (19.0,12.0) -- (18.0,11.0);}
{\draw[fill] (20.0,8.0) circle (2pt);}
{\draw[fill] (20.0,7.0) circle (2pt);}
{\draw (20.0,8.0) -- (20.0,7.0);}
{\draw[fill] (20.0,6.0) circle (2pt);}
{\draw (20.0,7.0) -- (20.0,6.0);}
{\draw[fill] (20.0,5.0) circle (2pt);}
{\draw (20.0,6.0) -- (20.0,5.0);}
{\draw[fill] (19.0,7.0) circle (2pt);}
{\draw (20.0,8.0) -- (19.0,7.0);}
{\draw[fill] (19.0,6.0) circle (2pt);}
{\draw (19.0,7.0) -- (19.0,6.0);}
{\draw (20.0,7.0) -- (19.0,6.0);}
{\draw[fill] (19.0,5.0) circle (2pt);}
{\draw (19.0,6.0) -- (19.0,5.0);}
{\draw (20.0,6.0) -- (19.0,5.0);}
{\draw[fill] (19.0,4.0) circle (2pt);}
{\draw (19.0,5.0) -- (19.0,4.0);}
{\draw (20.0,5.0) -- (19.0,4.0);}
{\draw[fill] (18.0,6.0) circle (2pt);}
{\draw (19.0,7.0) -- (18.0,6.0);}
{\draw[fill] (18.0,5.0) circle (2pt);}
{\draw (18.0,6.0) -- (18.0,5.0);}
{\draw (19.0,6.0) -- (18.0,5.0);}
{\draw[fill] (18.0,4.0) circle (2pt);}
{\draw (18.0,5.0) -- (18.0,4.0);}
{\draw (19.0,5.0) -- (18.0,4.0);}
{\draw[fill] (18.0,3.0) circle (2pt);}
{\draw (18.0,4.0) -- (18.0,3.0);}
{\draw (19.0,4.0) -- (18.0,3.0);}
{\draw[fill] (20.0,0.0) circle (2pt);}
{\node[above right=0pt,font=\tiny] at (12.0,14.0) {$1$};}
{\draw[fill] (12.0,14.0) circle (4pt);}
{\node[above right=0pt,font=\tiny] at (18.0,17.0) {$v_n$};}
{\draw[fill] (18.0,17.0) circle (4pt);}
{\node[above right=0pt,font=\tiny] at (6.0,11.0) {$v_n^{-1}$};}
{\draw[fill] (6.0,11.0) circle (4pt);}
{\node[above right=0pt,font=\tiny] at (10.5,13.0) {$\rho$};}
{\draw[fill] (11.0,13.0) circle (4pt);}
{\node[above right=0pt,font=\tiny] at (12.0,13.0) {$\tau$};}
{\draw[fill] (12.0,13.0) circle (4pt);}
{\node[above right=0pt,font=\tiny] at (6.1,7.7) {$\rho^{2^{n+1}-2}$};}
{\draw[fill] (6.0,8.0) circle (4pt);}
{\node[above right=0pt,font=\tiny] at (10.8,9.7) {$\{\tau^{2^{n}-1}, \rho^{2^{n+1}-2}v_n \}$};}
{\draw[fill] (12.0,11.0) circle (4pt);}
{\node[above right=0pt,font=\tiny] at (12.0,6.0) {$\tau^{2^{n+1}}$};}
{\draw[fill] (12.0,6.0) circle (4pt);}
{\draw[fill] (0.0,5.0) circle (6pt);}
{\draw[fill] (6.0,8.0) circle (6pt);}
{\draw[fill] (6.0,0.0) circle (6pt);}
{\draw[fill] (12.0,11.0) circle (6pt);}
{\draw[fill] (12.0,3.0) circle (6pt);}
{\draw[fill] (18.0,14.0) circle (6pt);}
{\draw[fill] (18.0,6.0) circle (6pt);}
        \end{tikzpicture}

  \caption{$K_{\star}(n)$ of the real numbers for $n=2$. The pattern is repeated horizontally by multiplication by $v_n$ and vertically by multiplication by $\tau^{2^{n+1}}$.}
  \label{fig:K(n)}
\end{figure}
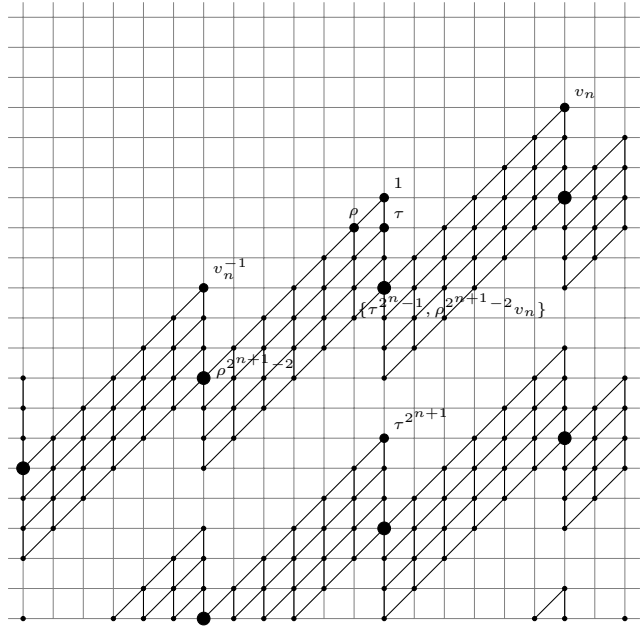

Below are some lemmas on the action of $\AAA^\star$ on $H^\star(F;\Z/2)$ used in the proof of \Cref{lem:Kn-coeff}.
The following lemma for $k = 2^j$ is \cite[Lemma 6.3]{Yagita:atiyah}.
\begin{lemma}
\label{lem:Sqtau}
The action of the motivic Steenrod squares on $\tau^n$ is given by
\begin{align*}
\Sq^{2k}(\tau^n) &= {\lfloor\frac{n+k-1}{2}\rfloor \choose k}\rho^{2k}\tau^{n-k}, k > 0, \\
\Sq^{2k+1}(\tau^n) &=
\left({\lfloor\frac{n+k+1}{2}\rfloor \choose k + 1}
+ {\lfloor\frac{n+k}{2}\rfloor \choose k + 1}\right)\rho^{2k+1}\tau^{n-k-1}.
\end{align*}
Here we define $\tau^j = 0$ for $j < 0$.
\end{lemma}
\begin{proof}
We use the notation $P^k = Sq^{2k}$ and $\Sq^{2k+1} = B^k$,
and define $P^k(\tau^n) = P^{k,n}\rho^{2k}\tau^{n-k}$, where $P^{k,n}\in\Z/2$, and similarly for $B^{k,n}\in\Z/2$.
The Cartan formula \cite[Proposition 9.6]{Voevodsky:power}, \cite[Proposition 4.4.2]{Riou}, \cite[p.~32]{Riou:slides} implies
\begin{align*}
P^{k}(\tau^n) &= \tau P^{k}(\tau^{n-1}) + \rho\tau B^{k-1}(\tau^{n-1}) \\
B^{k}(\tau^n) &= \tau B^{k}(\tau^{n-1}) + \rho P^{k-1}(\tau^{n-1}) + \rho^2 B^{k-1}(\tau^{n-1})
\end{align*}
which gives the recursion \eqref{eq:PBrecursion} in \Cref{lem:PBrecursion} with the solution \eqref{eq:PBsolution}.
\end{proof}

\begin{lemma}
\label{lem:PBrecursion}
Modulo 2
\begin{align}
\label{eq:PBsolution}
P^{k,n} &= \begin{cases}
1 & k = n = 0 \\
{\lfloor\frac{n+k-1}{2}\rfloor \choose k} & \text{otherwise},
\end{cases} \\
B^{k,n} &= {\lfloor\frac{n+k+1}{2}\rfloor \choose k + 1}
+ {\lfloor\frac{n+k}{2}\rfloor \choose k + 1},
\nonumber
\end{align}
solves the recursion
\begin{align}
\label{eq:PBrecursion}
P^{0,n} &= 1, B^{0,2n} = 0, B^{0,2n+1} = 1, n \geq 0, \\
P^{0,0} &= 1, B^{k,0} = 0, P^{k+1,0} = 0, k > 0, \nonumber\\
P^{k,n} &= P^{k,n-1} + B^{k-1,n-1}, \nonumber\\
B^{k,n} &= B^{k,n-1} + P^{k,n-1} + B^{k-1,n-1},\nonumber
\end{align}
for $k, n \geq 0$.
\end{lemma}
\begin{remark}
By adding correction terms of the form ${0 \choose k}{n \choose 0}$ to $P^{k,n}$ and $B^{k,n}$ in \eqref{eq:PBsolution} we obtain closed forms for $P^{k,n}$ solving \eqref{eq:PBrecursion} for all integers $k, n$.
\end{remark}

As a consequence of \Cref{lem:Sqtau} we get the following lemma.
This is also stated in \cite[Lemma 6.2]{Yagita:atiyah}.
\begin{lemma}
\label{cor:Qn-action}
The action of the Milnor-primitives on $\tau^n$ is given by
\[
Q_k(\tau^n) = {n \choose 2^k}\rho^{2^{k+1}-1}\tau^{n-2^k}.
\]
\end{lemma}
\begin{proof}
Do induction on $k$ and $n$, use \cite[Corollary 4]{Kylling} and the identity
\begin{align*}
{n \choose 2^{k+1}} =&
{n \choose 2^{k}}{\lfloor\frac{n-1}{2}\rfloor \choose 2^{k}}
+
{n-2^k \choose 2^{k}}{\lfloor\frac{n+2^k-1}{2}\rfloor \choose 2^{k}}\\
&+
{n-2^k \choose 2^{k}}{n \choose 2^{k-1}}{\lfloor\frac{n-1}{2}\rfloor \choose 2^{k-1}} \\
=&
{n \choose 2^{k}}{\lfloor\frac{n-1}{2}\rfloor \choose 2^{k}}
+
{n-2^k \choose 2^{k}}{\lfloor\frac{n+2^k-1}{2}\rfloor \choose 2^{k}} \bmod 2.
\end{align*}
The identity is easy to show using that ${n \choose 2^k} = 1$ if and only if the $k$th bit in the binary expansion of $n$ is $1$, and similar expressions for the other binomial coefficients.
\end{proof}

\printbibliography

\end{document}